\documentclass[11pt]{article}
\usepackage{epsfig,multicol,bbm,amsmath,amssymb,amscd,mathrsfs,fancybox,amsthm,framed,enumerate,booktabs}
\usepackage{fancyhdr}

\newcommand{\im}{\mathrm{Im}\,}
 \makeatletter
    
    \@addtoreset{equation}{section}
  \makeatother 
\usepackage[top=20truemm,bottom=20truemm,left=10truemm,right=10truemm]{geometry}

\newtheorem{Def}{Definition}[section]
\newtheorem{Thm}{Theorem}[section]
\newtheorem{Lem}{Lemma}[section]

\newtheorem{Ass}{Assumption}[section]
\newtheorem{Cor}{Corollary}[section]
\newtheorem{Rem}{Remark}[section]
\newtheorem{Hypo}{Hypothesis}[section]

\newcommand{\beq}{\begin{align}}
\newcommand{\eeq}{\end{align}}
\newcommand{\supp}{\mathrm{supp\,}}

\newcommand{\Expect}[1]{\left\langle{#1}\right\rangle}

\newcommand{\no}{\nonumber}

\newcommand{\Normal}[1]{ : \hspace{-0.5mm} {#1} \hspace{-0.6mm} : }

\newcommand{\Natural}{\mathbb{N}}

\newcommand{\Real}{\mathbb{R}}

\newcommand{\Complex}{\mathbb{C}}

\newcommand{\sgn}{\mathrm{sgn}\,}

\newcommand{\norm}[1]{\left\| #1 \right\|} 
\newcommand{\e}{\varepsilon}

\newcommand{\q}{\quad}
\renewcommand{\H}{\mathcal{H}}

\newcommand{\D}{\mathcal{D}}
\newcommand{\F}{\mathcal{F}}
\newcommand{\E}{\mathcal{E}}

\DeclareMathOperator*{\op}{\oplus}

\DeclareMathOperator*{\ot}{\otimes}
\DeclareMathOperator*{\hot}{\hat{\otimes}}
\DeclareMathOperator*{\wg}{\wedge}
\newcommand{\R}{\mathbb{R}}
\newcommand{\C}{\mathbb{C}}
\newcommand{\N}{\mathbb{N}}

\newcommand{\B}{\mathcal{B}}
\newcommand{\kk}{{\boldsymbol{k}}}
\newcommand{\pp}{{\boldsymbol{p}}}
\newcommand{\xx}{{\boldsymbol{x}}}

\newcommand{\yy}{{\boldsymbol{y}}}

\newcommand{\bb}{\mathrm{b}}
\newcommand{\ff}{\mathrm{f}}
\newcommand{\ee}{\mathbf{e}}
\newcommand{\dd}{\mathrm{d}}
\newcommand{\Ga}{\Gamma}
\newcommand{\La}{\Lambda}
\newcommand{\la}{\lambda}

\renewcommand{\rm}{\mathrm}

\renewcommand{\hat}{\widehat}

\newcommand{\fin}{\mathrm{fin}}
\newcommand{\chisp}{\chi_\mathrm{sp}}
\newcommand{\chiel}{\chi_\mathrm{el}}
\newcommand{\chiph}{\chi_\mathrm{ph}}
\newcommand{\pa}{\partial}

\title{Time-ordered exponential on the complex plane and Gell-Mann -- Low formula 
as a mathematical theorem}
\author{Shinichiro Futakuchi and Kouta Usui}
\begin{document}
\maketitle
\abstract{The time-ordered exponential representation of a complex time evolution operator in the interaction picture is studied. Using the complex time evolution, we prove the Gell-Mann -- Low formula under certain abstract conditions,
in mathematically rigorous manner. 
We apply the abstract results to quantum electrodynamics with cutoffs.}

\section{Introduction}

In this paper, we consider a formula in quantum field theories of the type
\begin{align}
\Expect{\Omega , T \big\{ \phi ^{(1)} (x_1) \cdots \phi ^{(n)} (x_n) \big\} \Omega } 
& = \lim _{ t\to \infty } \frac{ \Expect{\Omega _0 , T \big\{ \phi _\rm{I} ^{(1)} (x_1) \cdots \phi ^{(n)} _\rm{I} (x_n) \rm{exp} \big[ -i \int _{-t} ^t d\tau H_1 (\tau ) \big] \big\} \Omega _0 } }{\Expect{ \Omega _0 , T \big\{ \rm{exp} \big[ -i \int _{-t} ^t d\tau H_1 (\tau ) \big] \big\} \Omega _0 }} , \label{introGML}
\end{align}
called the \textit{Gell-Mann -- Low formula} \cite{MR0044395}. 
The meaning of each symbol in the formula \eqref{introGML} is as follows: the symbol $ \Expect{\cdot \, , \cdot } $ denotes the inner product of a Hilbert space of quantum state vectors, 
$ \phi ^{(k)} (x_k) $ and $ \phi ^{(k)} _\rm{I} (x_k) \; (k=1,...,n, \; x_k \in \R ^4) $ denote field operators
 in the Heisenberg and the interaction picture, respectively. For instance, in quantum electrodynamics (QED), 
 each $ \phi ^{(k)} $ denotes the Dirac field $ \psi _l $, its conjugate $ \psi _l ^\dagger $, or the gauge field $ A_\mu $.
 The symbol $ T $ denotes the time-ordering and $ \Omega $ and $ \Omega _0 $ the vacuum states of the interacting and the free theory, respectively. The operator
 \[  T \big\{ \rm{exp} \big[ -i \int _{-t} ^t d\tau H_1 (\tau ) \big] \big\} \]
is the time evolution operator in the interaction picture, having the following series expansion:
\begin{align}
& T \big\{ \rm{exp} \big[ -i \int _{-t} ^t d\tau H_1 (\tau ) \big] \big\} \no \\
& \q = 1+ (-i) \int _{-t} ^t d\tau _1 \, H_1 (\tau _1) + (-i) ^2 \int _{-t} ^t d\tau _1 \int _{-t} ^{\tau _1}  d\tau _2 \, H_1 (\tau _1) H_1 (\tau _2) + \cdots ,
\end{align}
which is often called the \textit{time-ordered exponential} or the \textit{Dyson series} for $ H_1 (\tau ) := e^{i\tau H_0} H_1 e^{-i\tau H_0} \; (\tau \in \R ) $, where $ H_0 $ and $ H_1 $ are the free and the interaction Hamiltonians.

This formula is a fundamental tool to generate a perturbative expansion 
of the \textit{$ n $-point correlation function} 
\[  \Expect{\Omega , T \big\{ \phi ^{(1)} (x_1) \cdots \phi ^{(n)} (x_n) \big\} \Omega } \]
with respect to the coupling constant. When the coupling is small enough (for QED, this seems valid), 
the first few terms of the perturbation series is expected to be a
good approximation of the correlation function which gives quantitative predictions for 
 observable variables such as scattering cross section. In QED, these predictions agree with experimental results to
 eight significant figures, the most accurate predictions in all of natural
 science. However, the mathematical derivation of \eqref{introGML} is far from trivial and 
 proofs given in physics literatures are very heuristic and informal. In fact, 
 even the Hamiltonian is not easily given a mathematical meaning. 
The purpose of the present paper is to construct a mathematically rigorous setup in which 
the Gell-Mann -- Low formula \eqref{introGML} is adequately formulated and proved. 

Even when the $ n $-point correlation function 
\[  \Expect{\Omega , T \big\{ \phi ^{(1)} (x_1) \cdots \phi ^{(n)} (x_n) \big\} \Omega } \]
does not mathematically make sense, we can formally compute (after a renormalization procedure) 
this quantity via formal perturbation series to arbitrary order of the coupling constant and
it is this computation that agrees with experiments extremely well. 
Such formal computations can not be regarded as an approximation of the $n$-point function
without a mathematical meaning of it,
but should be regarded as a \textit{definition} of the $n$-point function through the perturbation series.
Hence, what is lacking is the knowledge about what quantity is approximated by the perturbation series
and about the relation between the ordinary Hilbert-space formulation of quantum theory and
the perturbation series. In other words, we have to clarify in what sense 
a perturbative formulation of quantum field theory is indeed
a ``quantum" theory.
Thus, it is very important in mathematical and \textit{physical} point of view to study under what conditions the
Gell-Mann -- Low formula \eqref{introGML} is indeed true as a mathematical theorem
within a Hilbert space formulation of quantum theory.

In the 1960s, Wightman and G{\aa}rding \cite{GWaxioms} formulated a set of axioms in the framework of quantum mechanics which requires minimum properties that relativistic quantum field theory should satisfy. However, it is extremely difficult to construct a non-trivial model in the four-dimensional space-time which is physically acceptable and fulfills the axioms, and no such model has been found so far. 
We do not intend to construct such ideal models but abandon some of the axioms by introducing 
several regularizations so that each object is easily given mathematical meaning (of course, 
regularizations are employed in such a way that all the objects heuristically tends to the ideal ones 
in the limit where the regularizations are removed).  
In this way, field operators and a Hamiltonian is rigorously defined
as linear operators acting in some Hilbert space. Furthermore, the vacuum states $ \Omega $ and $ \Omega _0 $ are 
realized as the eigenvectors corresponds to the infimum of the spectrum of the total and free Hamiltonians, if these exist. The existence of the ground state $ \Omega $, on which the validity of the
Gell-Mann -- Low formula crucially depends, is far from trivial, because it needs to analyze the perturbation of eigenvalues embedded in the continuous spectrum, to which regular perturbation theory \cite{MR1335452} 
can not be applied. From the late 1990s to the 2000s, several important methods to prove the existence of ground states were developed in the study of a quantum system consisting of quantum particles and a Bose field (for example, see \cite{MR1491549,MR1724854,MR1777307,MR1856401,MR1623746}). These methods have been improved 
by many authors to be also applicable to systems of interacting quantum fields \cite{MR2354352,MR2070129,MR2559073,MR1983728,MR2541206,MR2813489,1405.3773,MR1914143}. 
Once field operators and the ground state are given, we can define the $ n $-point correlation function 
\[ \Expect{\Omega , T \big\{ \phi ^{(1)} (x_1) \cdots \phi ^{(n)} (x_n) \big\} \Omega } \]
\textit{non-perturbatively}. The proof of the Gell-Mann -- Low formula is the first step
to reveal the relation between the series expansion (which may be divergent asymptotic series) of
the non-perturbatively defined objects in this way and the formal perturbation series given in physics literatures.

In the heuristic proof of \eqref{introGML}, Murray Gell-Mann and Francis Low \cite{MR0044395} introduced \textit{adiabatic switching} of the interaction through the time-dependent Hamiltonian of the form $ H_0 + e^{-\e |t| } H_1 $, where $ \e >0 $ is the small parameter which eventually vanishes. We take an alternative way by sending the time $ t $ to $ \infty $ in the imaginary direction: $ t \to \infty ( 1-i \e ) $. The same method can be found in physics literatures (see, for example, \cite{MR1402248,MR2257528}). In this case, one difficulty with the mathematical proof of \eqref{introGML} is to construct the complex time evolution which possesses the following series expansion:
\begin{align}
T \big\{ \rm{exp} \big[ -i \int _{z'} ^z d\zeta H_1 (\zeta ) \big] \big\} 
& \q = 1+ (-i) \int _{z'} ^z d\zeta _1 \, H_1 (\zeta _1) + (-i) ^2 \int _{z'} ^z d\zeta _1 \int _{z'} ^{\zeta _1}  d\zeta _2 \, H_1 (\zeta _1) H_1 (\zeta _2) + \cdots ,
\end{align}
$(z,z' \in \C )$. If $ H_1(\zeta ) \; (\zeta \in \C ) $ are bounded operators, it is easy to see that the integrals on the right-hand side can be taken in the sense of line integral and the series converges absolutely under some suitable conditions, but these are unbounded operators in most cases. In the previous paper \cite{FutakuchiUsui2013}, the authors investigated the time-ordered exponential for unbounded operators only in the real time. In this paper, we extend the methods obtained in \cite{FutakuchiUsui2013} to the complex time. 

The outline of the present paper is as follows. In Section 2, we develop an abstract theory of complex 
time-ordered exponential. In Section 3, we state and prove the 
Gell-Mann -- Low formula in an abstract form under some assumptions. 
In Section 4, we apply our abstract results to QED.

\section{Abstract construction of time-ordered exponential on the complex plane and its properties}

Let $ \H $ be a complex Hilbert space. The inner product and the norm of $ \H $ are denoted by $ \Expect{ \cdot , \cdot } _\H $ (anti-linear in the first variable) and $ \| \cdot \| _\H $ respectively. When there can be no danger of confusion, then the subscript $ \H $ in $ \Expect{ \cdot , \cdot } _\H $ and $ \| \cdot \| _\H $ is omitted. For a linear operator $ T $ in $ \H $, we denote its domain (resp. range) by $ D(T) $ (resp. $ R(T) $). We also denote the adjoint of $T$ by $ T^* $ and the closure by $ \bar{T} $ if these exist. For a self-adjoint operator $ T $,  $ E_T (\cdot ) $ denotes the spectral measure of $ T $. The symbol $T|_D$ denotes the restriction of a linear operator $T$ to the subspace $D$. For a linear operators $ S $ and $ T $ on a Hilbert space, $ D(S+T) := D(S) \cap D(T) , \; D(ST) := \{ \Psi \in D(T) \, | \, T \Psi \in D(S) \} $ unless otherwise stated.

We begin by defining a time-ordered product of operator-valued functions and the
time-ordered exponential of an operator-valued function in an unambiguous way.
Let $z,z'\in\Complex$ and $\Gamma$ be a piecewisely continuously differentiable simple curve in $\C$
from $z'$ to $z$. That is, $\Gamma$ is a map from a closed interval $I=[\alpha ,\beta ]$ in $\R$ into $\C$,
which is piecewisely continuously differentiable and injective, satisfying 
\begin{align}
\Gamma( \alpha )=z',\quad \Gamma( \beta )=z.
\end{align}
We define a linear order $\succ$ on $\Gamma(I) = \{ \Gamma (t) \; | \; t\in I \} \subset \C $ as follows. For $\zeta_1,\zeta_2 \in \Gamma (I) $, there exist $t_1, t_2\in I$ with
$\Gamma(t_1)=\zeta_1$ and $\Gamma(t_2)=\zeta_2$.
Then, $\zeta_1\succ\zeta_2$ if and only if $t_1>t_2$. 

In what follows, we denote $ \Gamma (I) $ simply by $ \Gamma $. Let $\mathfrak{S}_{n}$ be the symmetric group of order $n\in\N$ and $L(\H)$ be (not necessarily bounded) linear operators in $\H$. 
For mappings $F_1,F_2,\dots,F_k \; (k \in \N ) $ from $\Gamma$ into $L(\H)$,
we define a map $ T[F_1 \dots F_k] $ from $\Gamma^k $ into $ L(\H) $ by 
\begin{align}\label{Teitoku}
T[F_1\dots F_k](\zeta_1,\dots,\zeta_k):=\sum_{\sigma\in\mathfrak{S}_k}\chi_{P_\sigma}(\zeta_1,\dots,\zeta_k)
F_{\sigma(1)}(\zeta_{\sigma(1)})\dots F_{\sigma(k)}(\zeta_{\sigma(k)}),
\end{align}
whenever the right-hand side makes sense, where $\chi_J$ denotes the 
characteristic function of the set $J$, and
\begin{align}
P_\sigma = \{(\zeta_1,\dots,\zeta_k)\in\Gamma^k\,|\, \zeta_{\sigma(1)}\succ \dots \succ\zeta_{\sigma(k)}\}, \q \sigma\in\mathfrak{S}_k .
\end{align} 
In what follows, we sometimes adopt a little bit confusing notation
\begin{align}
T\left(F_1(\zeta_1)\dots F_k(\zeta_k)\right):=T[F_1\dots F_k](\zeta_1,\dots,\zeta_k) ,
\end{align} 
and call it a \textit{time-ordered product} of $F_1(\zeta_1),\dots, F_k(\zeta_k)$, even though
the operation $T$ does \textit{not} act on the product of operators $F_1(\zeta_1),\dots, F_k(\zeta_k)$
but on the product of mappings $F_1,\dots,F_k$.

Next, we define time-ordered exponential of an operator-valued function. Let $F:\Gamma\to L(\H)$ and
let $C(F)\subset \H$ be a linear subspace spanned by all the vectors $\Psi\in\H$
such that the mapping
\begin{align}
(\zeta_1,\dots,\zeta_n)\mapsto F(\zeta_1)\dots F(\zeta_n)\Psi 
\end{align}
is strongly continuous on some region containing $\Gamma^n$. We define a time-ordered exponential operator
by 
\begin{align}\label{TO}
D\left(T\exp \left(\int_\Gamma d\zeta \,F(\zeta) \right)\right)&:=\left\{\Psi\in C(F)\,\left|\,\sum_{n=0}^\infty \frac{1}{n!}\norm{\int_{\Gamma^n}d\zeta_1\dots d\zeta_n\,
T\left(F(\zeta_1)\dots F(\zeta_n)\right)\Psi}<\infty\right.  \right\},\\
T\exp \left(\int_\Gamma d\zeta \,F(\zeta) \right)\Psi&:=\sum_{n=0}^\infty \frac{1}{n!}\int_{\Gamma^n}d\zeta_1\dots d\zeta_n\,
T\left(F(\zeta_1)\dots F(\zeta_n)\right)\Psi,
\end{align}
where the integration is understood in the strong sense. 

We also define a more general time-ordered exponential operator.
Let $F_1,F_2,\dots,F_k,\dots,F_{k+n}$ be the mappings from $\Gamma$ into (not necessarily bounded)
liner operators in $\H$. We define a map from $\Gamma^n$ into $L(\H)$,
which is labeled by $(\zeta_1,\dots,\zeta_k)\in\Gamma^k$, 
\begin{align}
T[F_1(\zeta_1)F_2(\zeta_2)\dots F_k(\zeta_k)F_{k+1}\dots F_{k+n}]:\Gamma^n \to L(\H)
\end{align}
by the relation
\begin{align}
T[F_1(\zeta_1)F_2(\zeta_2)\dots F_k(\zeta_k)F_{k+1}\dots F_{k+n}](\zeta_{k+1},\dots,\zeta_{k+n}):=\sum_{\sigma\in\mathfrak{S}_{k+n}}\chi_{P'_{n,\sigma}}(\zeta_{k+1},\dots,\zeta_{k+n})
F_{\sigma(1)}(\zeta_{\sigma(1)})\dots F_{\sigma(k+n)}(\zeta_{\sigma(k+n)}),
\end{align}
whenever the operator products on the right-hand side makes sense. Here, we denote
\begin{align}
P'_{n,\sigma}:=\{(\zeta_{k+1},\dots,\zeta_{k+n})\in\Gamma^n\,|\, \zeta_{\sigma(1)}\succ \dots \succ\zeta_{\sigma(k+n)}\}
\end{align}
for $\sigma\in\mathfrak{S}_{k+n}$.
In this case, we also employ a confusing notation (really confusing in the case)
\begin{align}
  T\left(F_1(\zeta_1)\dots F_{k+n}(\zeta_{k+n})\right):=T[F_1(\zeta_1)F_2(\zeta_2)\dots F_k(\zeta_k)F_{k+1}\dots F_{k+n}](\zeta_{k+1},\dots,\zeta_{k+n}),
\end{align} 
and call it a \textit{time-ordered product} of $F_1(\zeta_1),\dots, F_{k+n}(\zeta_{k+n})$, 
following physics literatures. We never use this notation unless it can be clearly understood
 from a context which variables of $(\zeta_1,\dots,\zeta_{k+n})$ are fixed and
which variables are function argument.

Using this notation, we can define more general time-ordered exponential operator.
Let $F_1,\dots,F_k,F$ be operator-valued functions from $\Gamma $ into $ L(\H)$ and 
$F_{k+1}=\dots = F_{k+n}=F$. Let $C(F_1,\dots,F_k,F)$
 be a linear subspace spanned by all the vectors $\Psi$ for which the mappings
 \begin{align}
 (\zeta_{k+1},\dots,\zeta_{k+n})\mapsto F_{\sigma(1)}(\zeta_{\sigma(1)})\dots F_{\sigma(k+n)}(\zeta_{\sigma(k+n)})\Psi
 \end{align}
 are continuous for all fixed $(\zeta_1,\dots,\zeta_k)$ and all $\sigma\in\mathfrak{S}_{n+k}$. 
 Then, on the domain
\begin{align}\label{TO}
&D\left(TF_1(\zeta_1)\dots F_k(\zeta_k)\exp \left(\int_\Gamma d\zeta \,F(\zeta) \right)\right)\no\\
&\qquad\qquad:=\left\{\Psi\in C(F_1,\dots,F_k,F)\,\left|\,\sum_{n=0}^\infty \frac{1}{n!}\norm{\int_{\Gamma^n}d\zeta_{k+1}\dots 
d\zeta_{k+n}\,
T\left(F_1(\zeta_1)\dots F_k(\zeta_k)F(\zeta_{k+1})\dots F(\zeta_{k+n})\right)\Psi}<\infty\right.  \right\},
\end{align}
We define
\begin{align}
TF_1(\zeta_1)\dots F_k(\zeta_k)\exp \left(\int_\Gamma d\zeta \,F(\zeta) \right)\Psi&:=\sum_{n=0}^\infty \frac{1}{n!}\int_{\Gamma^n}d\zeta_{k+1}\dots d\zeta_{k+n}\,
T\left(F_1(\zeta_1)\dots F_k(\zeta_k)F(\zeta_{k+1})\dots F(\zeta_{k+n})\right)\Psi.
\end{align}
We remark that for all $\sigma\in\mathfrak{S}_k$,
\begin{align}
TF_1(\zeta_1)\dots F_k(\zeta_k)\exp \left(\int_\Gamma d\zeta \,F(\zeta) \right)=
TF_{\sigma(1)}(\zeta_{\sigma(1)})\dots F_{\sigma(k)}(\zeta_{\sigma(k)})\exp \left(\int_\Gamma d\zeta \,F(\zeta) \right).
\end{align} 

We introduce a class of operators which plays a crucial role in the following analyses.
Let $H_0$ be a non-negative self-adjoint operator in $\H$.
\begin{Def}[$ \mathcal{C}_0 $-class]\label{C_0}\normalfont 
We say that a linear operator $ T $ is in \textit{$ \mathcal{C}_0 $-class} if $ T $ satisfies the following (I)-(III):
\begin{enumerate}[(I)]
\item $T$ and $T^*$ are densely defined and closed. 
\item $ T $ and $T^*$ are $ H_0^{1/2}$-bounded.
\item There exists a constant $ b \ge 0 $ such that, for all $ E \ge 0 $, $ T $ and $T^*$ map $ R (E_{H_0} ([0, E]) ) $ into $ R (E_{H_0} ([0, E+b]) ) $.
\end{enumerate}
\end{Def}
We define
\begin{align}
& V_E := R (E_{H_0} ([0, E]) ) , \\
& D_\rm{fin} := \bigcup _{ E \ge 0 } V_E ,
\end{align}
and denote the set consisting of all the $\mathcal{C}_0$ class operators also by $\mathcal{C}_0$.
Note that the subspace $ D_\rm{fin} $ is dense in $ \H $ since $ H_0 $ is self-adjoint. 
For $A\in\mathcal{C}_0$, we denote
\begin{align}
A(z):=e^{izH_0}A e^{-izH_0} , \q z \in \C .
\end{align}
Note that $A(z)$ is closable since its adjoint includes the operator
$e^{iz^*H_0}A^*e^{-iz^*H_0}$ which is densely defined.
We denote the closure of $A(z)$ by the same symbol.
In this notation, one obtains
\begin{align}
A(z)^* \supset A^*(z^*).
\end{align}

The goal of the present section is to prove following Theorems \ref{time-ordered-exp}-\ref{Time-order-ext}.
\begin{Thm}\label{time-ordered-exp} Let $A$ be in $\mathcal{C}_0$ class and $z,z'\in\C$. 
\begin{enumerate}[(i)]
\item Take a piecewisely continuously differentiable simple curve $\Gamma_{z,z'}$
which starts at $z'$ and ends at $z$ with $\im z'\le\im z$. Then,
\begin{align}
D_\fin\subset D\left(T\exp\left(-i\int_{\Gamma_{z,z'}} d\zeta A(\zeta)\right)\right)
\end{align}
and the restriction 
\begin{align}
T\exp\left(-i\int_{\Gamma_{z,z'}} d\zeta A(\zeta)\right)\Big|_{D_\fin}
\end{align}
does not depend upon the simple curve from $z'$ to $z$ and
depends only on $z$ and $z'$, justifying the notation 
\begin{align}
U(A;z,z'):=T\exp\left(-i\int_{\Gamma_{z,z'}} d\zeta A(\zeta)\right)\Big|_{D_\fin}.
\end{align}
\item $ U(A;z,z') $ is closable, and satisfies the following inclusion relation:
\begin{align}\label{adjoint2}
U (A;z,z') ^* \supset \overline{U(A^*;z' {}^* ,z^*)} .
\end{align}
\end{enumerate}
\end{Thm}

\begin{Lem}\label{analytic}
Let $A_1,\dots, A_n$ be in $\mathcal{C}_0$-class. Then, for all $\Psi\in D_\fin$ and all $n\in\Natural$,
the mapping
\begin{align}
\C^n\ni (z_1,\dots,z_n)\mapsto A_1(z_1)\dots A_n(z_n)\Psi\in\H
\end{align}
is strongly analytic in $\C^n$. 
\end{Lem}
\begin{proof}
Each vector in $D_\fin$ is an entire analytic vector of $H_0$, and 
each $A_j\in\mathcal{C}_0$ $(j=1,2,\dots,n)$ preserves the subspace of all the entire analytic vectors of $H_0$.
Therefore, $A_1(z_1)\dots A_n(z_n)\Psi$ permits an absolutely converging power series expansion in $z_1,\dots,z_n$
and thus is strongly analytic. 
\end{proof}

From Lemma \ref{analytic}, we can define a liner operator $V_n(A;z,z')$ with 
the domain $D(V_n(A;z,z'))=D_\fin$ for $A\in\mathcal{C}_0$, $z,z'\in \C$, $n\in\N$, and $\Psi\in D_\fin$,
\begin{align}
V_n(A;z,z')\Psi:=\frac{(-i)^n}{n!}\int_{\Gamma_{z,z'}^n} d\zeta_1\dots d\zeta_n\,T\left(A(\zeta_1)\dots A(\zeta_n)\right)\Psi, 
\end{align}
where $\Gamma$ denotes a piecewisely continuously differentiable simple curve 
from $z'$ to $z$. We regard $V_0(A;z,z')=1$.

\begin{Lem}\label{int-rep-and-recursion}
\begin{enumerate}[(i)]
\item If $\Psi\in V_E$, then $V_n(A;z,z')\Psi\in V_{E+nb}$, where $b\ge 0$ is a constant stated in Definition \ref{C_0} (III)
for $A\in\mathcal{C}_0$.
\item The operator $V_n(A;z,z')$ has the following representation
\begin{align}
V_n(A;z,z')&=(-i)^n\int_{z'}^z d\zeta_1\int_{z'}^{\zeta_1} d\zeta_2\dots \int_{z'}^{\zeta_{n-1}} d\zeta_n\, A(\zeta_1) A(\zeta_2) \dots A(\zeta_n)\\
&=(-i)^n\int_{z'}^z d\zeta_n\int_{\zeta_n}^z d\zeta_{n-1}\dots \int_{\zeta_2}^z d\zeta_1\,A(\zeta_1) A(\zeta_2) \dots A(\zeta_n).
\end{align}
where the above integrations denote the indefinite integral of an analytic function which
depends only on the start and the end point.
\item $V_n(A;z,z')$ is analytic in $z\in\C$ and $z'\in\C$, and independent of the choice of a simple curve $\Gamma_{z,z'}$
from $z'$ to $z$.
\item $V_n(A;z,z')$ satisfies the formulae for $n=0,1,\dots$,
\begin{align}
V_{n+1}(A;z,z')&=(-i)\int_{z'}^z d\zeta \,A(\zeta) V_n(A;\zeta,z')\label{inteq1} \\
&=(-i)\int_{z'}^z d\zeta \,V_n(A;z,\zeta)A(\zeta) .  \label{inteq2}
\end{align} 
\end{enumerate}
\end{Lem}
\begin{proof}
The assertion (i) follows from the fact that 
\begin{align}
T\left(A(\zeta_1)\dots A(\zeta_n)\right)\Psi\in V_{E+nb}
\end{align}
and $V_{E+nb}$ is closed. Since (iii) and (iv) are
simple corollaries of (ii), it suffices to prove (ii).   
We prove only the case where $\Gamma:[\alpha,\beta]\to\C$ is continuously differentiable. A general case is
straightforward. 
By definition of the time-ordering operation $T$ \eqref{Teitoku}, one finds on $D_\fin$
\begin{align}\label{timed}
V_n(A;z,z')&=\frac{(-i)^n}{n!}\int_{\Gamma^n} d\zeta_1\dots d\zeta_n\,T\left(A(\zeta_1)\dots A(\zeta_n)\right)\no\\
&=\frac{(-i)^n}{n!}\sum_{\sigma\in\mathfrak{S_n}}
\int_{\{\beta\ge t_{\sigma(1)}>\dots> t_{\sigma(n)}\ge \alpha\}} dt_1\dots dt_n\, \Gamma'(t_1)\dots \Gamma'(t_n)\,A(\Gamma(t_{\sigma(1)}))\dots A(\Gamma(t_{\sigma(n)})).
\end{align}
The above integration does not
depend on $\sigma\in\mathfrak{S}_n$ and is equal to
\begin{align}
&\int_{\{\beta\ge t_1>\dots > t_n \ge \alpha\}} dt_1\dots dt_n\, \Gamma'(t_1)\dots \Gamma'(t_n)\,A(\Gamma(t_1))\dots A(\Gamma(t_n)) \no \\
&=\int_\alpha^\beta dt_1\Gamma'(t_1)\int_\alpha^{t_1} dt_2 \Gamma'(t_2)\dots \int_\alpha^{t_{n-1}} dt_n\Gamma'(t_n)\,A(\Gamma(t_1))\dots A(\Gamma(t_n)) \label{part1}\\
&=\int_\alpha^\beta dt_n\Gamma'(t_n)\int_{t_n}^{\beta} dt_{n-1} \Gamma'(t_{n-1})\dots \int_{t_2}^\beta dt_1\Gamma'(t_1)\,A(\Gamma(t_1)) 
\dots A(\Gamma(t_n)) . \label{part2}
\end{align}
The expression \eqref{part1} and \eqref{part2} can be rewritten
\begin{align}
\int_{z'}^z d\zeta_1\int_{z'}^{\zeta_1} d\zeta_2\dots \int_{z'}^{\zeta_{n-1}} d\zeta_n\,A(\zeta_1)\dots A(\zeta_n) 
\end{align}
and
\begin{align}
\int_{z'}^z d\zeta_n\int_{\zeta_n}^z d\zeta_{n-1}\dots \int_{\zeta_2}^z d\zeta_1\,A(\zeta_1)\dots A(\zeta_n) 
\end{align}
respectively.
Since the summation over $\sigma$ gives $n!$, the assertion (ii) follows.
\end{proof}

In the following, we employ the notation 
\begin{align}
\frac{(-i)^n}{n!}\sum_{\sigma\in\mathfrak{S_n}}
\int_{\{\zeta_{\sigma(1)}\succ\dots\succ\zeta_{\sigma(n)}\}} d\zeta_1\dots d\zeta_n\,A(\zeta_{\sigma(1)})\dots A(\zeta_{\sigma(n)}).
\end{align}
to denote the integration such as \eqref{timed}.

\begin{Lem}\label{boundLem}
For all $ n \ge 0 $, $A\in\mathcal{C}_0$, $ E \ge 0 $ and $ \Psi \in V_E $, the following estimate holds
for all $z,z'\in\C$ with $\im z\le\im z'$.
\begin{align}
\norm{ V_n (A;z,z') \Psi } \le C^n e^{|\im z'|(2E+nb)}\frac{|z-z'| ^n}{n!}\big( E+(n-1)b +1 \big) ^{1/2} \cdots \big( E +1 \big) ^{1/2}  \norm{ \Psi } ,
\end{align}
where $b\ge 0$ is a constant stated in Definition \ref{C_0} (III) and $C=\norm{A(H_0+1)^{-1/2}}$.
In the case where $ n=0 $, we regard the right-hand side as $ \| \Psi \| $.
\end{Lem}
\begin{proof}
First, we prove for $\im z_1\le \im z_2\le \dots \le \im z_n$,
\begin{align}\label{inner-bound}
\norm{A(z_1)\dots A(z_n)\Psi} \le C^ne^{|\im z_n|(2E+nb)}(E+(n-1)b+1)^{1/2}\dots(E+1)^{1/2}\norm{\Psi}.
\end{align}
In fact, the identity 
\begin{align}
A(z_1)\dots A(z_n)\Psi&=e^{iz_1H_0}Ae^{-i(z_1-z_2)H_0}\dots e^{-i(z_{n-1}-z_n)H_0} A_n e^{iz_nH_0}\Psi \no\\
&=e^{iz_1H_0}E_{H_0}([0,E+nb])A(H_0+1)^{-1/2}(H_0+1)^{1/2}E_{H_0}([0,E+(n-1)b])e^{-i(z_1-z_2)H_0}\times\dots \no\\
&\qquad\qquad \times e^{-i(z_{n-1}-z_n)H_0} A(H_0+1)^{-1/2}(H_0+1)^{1/2}E_{H_0}([0,E]) e^{iz_nH_0}\Psi 
\end{align}
implies \eqref{inner-bound}, because $e^{-i(z_j-z_{j+1})H_0}$ ($j=1,2,\dots,n-1$) are bounded
with operator norms less than $1$.
From Lemma \ref{int-rep-and-recursion} (iii), to estimate $\norm{V_n(A;z,z')}$ 
we can choose the path $ C $ from $ z' $ to $ z $ as 
\begin{align}
C(t) = z' + (z-z') t , \q t\in [0,1] .
\end{align}
Then, we have $ C' (t) := (d/dt) C(t) = z-z' $ and by \eqref{inner-bound}
\begin{align}
\norm{ V_n (A;z,z') \Psi }&\le \frac{1}{n!}C^n|z-z'|^n 
\int_{[0,1]^n}dt_1\dots dt_n\, e^{|\im z_n|(2E+nb)}(E+(n-1)b+1)^{1/2}\dots(E+1)^{1/2}\norm{\Psi} \no\\
&\le C^n e^{|\im z'|(2E+nb)}\frac{|z-z'| ^n}{n!}\big( E+(n-1)b +1 \big) ^{1/2} \cdots \big( E +1 \big) ^{1/2}  \norm{ \Psi }.\end{align}
This completes the proof.
\end{proof}

For $\zeta,\zeta'\in\C$ and $T\in\mathcal{C}_0$, we denote
\begin{align}
T(\zeta,\zeta'):=e^{i\zeta H_0}Te^{i\zeta'H_0}.
\end{align}
Note that 
\begin{align}
T(\zeta)=T(\zeta,-\zeta).
\end{align}

\begin{Lem}\label{boundLem2} 
Let $ T_k ,A_k\, (k= 1, ... , m , \; m \ge 1) $ be $ \mathcal{C}_0 $-class operators. Then, for all $ \Psi \in D_\rm{fin} $, $ z_k , z_k ' \in \C \; (k=1,...,m) $ with $\im z_k\le \im z'_k$ and $\zeta_k,\zeta_k'\in \C$, it follows that
\begin{align}\label{D-bound}
 \sum _{ n_1 , ... , n_m  = 0} ^\infty \| T_m(\zeta_m,\zeta_m') V_{n_m}(A_m;z_m ,z_m') \cdots T_1(\zeta_1,\zeta_1') V_{n_1}(A_1;z_1 , z_1') \Psi \| < \infty .
\end{align}
Furthermore, the convergence is locally uniform in $\zeta_1,\zeta_1',z_1,z'_1,\dots, \zeta_m,\zeta_m',z_m,z'_m$.
\end{Lem}
\begin{proof}  
Let $\Psi\in V_E$ and put for $k=1,2,\dots,m$,
\begin{align}
\Psi_k=T_k(\zeta_k,\zeta_k') V_{n_k}(A_k;z_k ,z_k') \cdots T_1(\zeta_1,\zeta_1') V_{n_1}(A_1;z_1 , z_1') \Psi.
\end{align}
Let $a_k,b_k\ge 0$ ($k=1,2,\dots,m$) be constants stated in Definition \ref{C_0} (III) regarding $T_k,A_k$, respectively.
We denote
\begin{align}
a=\max_k\{a_k\},\q b=\max_k\{b_k\}.
\end{align}
Then, we see from Lemma \ref{int-rep-and-recursion} (i) that
\begin{align}
\Psi_k\in V_{E+(n_1+\dots+n_{k})b+ka}.
\end{align}
Put
\begin{align}
K=\max_k\{|\im \zeta_k|,|\im \zeta_k'| ,|\im z'_k|\},\q N=n_1+\dots+n_m,\q C=\max_{k}\left\{ \norm{T_k(H_0+1)^{-1/2}}, \norm{A_k(H_0+1)^{-1/2}} \right\}.
\end{align}
Then, from Lemma \ref{boundLem}, we have
\begin{align}
&\norm{T_m(\zeta_m,\zeta_m') V_{n_m}(A_m;z_m ,z_m') \cdots T_1(\zeta_1,\zeta_1') V_{n_1}(A_1;z_1 , z_1') \Psi}\no\\
=& \norm{T_m(\zeta_m,\zeta_m') V_{n_m}(A_m;z_m ,z_m')\Psi_{m-1}} \no\\
\le& e^{2K(E+Nb+ma)}
(E+Nb+(m-1)a+1)^{1/2}\norm{T_m(H_0+1)^{-1/2}} \norm{V_{n_m}(A_m;z_m,z'_m)\Psi_{m-1}}\no\\
\le& e^{2K(E+Nb+ma)}C^{n_m+1}
e^{|\im z'_m|(2E+2(n_1+\dots+n_{m-1})b+2(m-1)a+n_m b)}\frac{|z_m-z'_m| ^{n_m}}{n_m!}
\times\no\\
&\qquad\qquad\times\big( E+Nb +(m-1)a+1 \big) ^{1/2} \cdots \big( 
E+(N-n_m)b +(m-1)a+1 \big) ^{1/2}  \norm{ \Psi_{m-1} }\no\\
\le& e^{4K(E+Nb+ma)}C^{n_m+1}
\frac{|z_m-z'_m| ^{n_m}}{n_m!}
\big( E+Nb +(m-1)a+1 \big) ^{1/2} \cdots \big( 
E+(N-n_m)b +(m-1)a+1 \big) ^{1/2}  \norm{ \Psi_{m-1} } .
\end{align}
Repeating this estimate, we arrive at
\begin{align}
&\norm{T_m(\zeta_m,\zeta_m') V_{n_m}(A_m;z_m ,z_m') \cdots T_1(\zeta_1,\zeta_1') V_{n_1}(A_1;z_1 , z_1') \Psi}\no\\
\le &e^{4mK(E+Nb+ma)}C^{N+m}\frac{|z_m-z'_m|^{n_m}\dots|z_1-z'_1|^{n_1}}{n_m!\dots n_1!}
\big( E+Nb +(m-1)a+1 \big) ^{1/2} \cdots \big( 
E+(m-1)a+1 \big) ^{1/2}  \norm{ \Psi }.
\end{align}
Therefore, we obtain
\begin{align}\label{final}
&\sum _{ n_1 , ... , n_m  = 0} ^\infty \| T_m(\zeta_m,\zeta_m') V_{n_m}(A_m;z_m ,z_m') \cdots T_1(\zeta_1,\zeta_1') V_{n_1}(A_1;z_1 , z_1') \Psi \| \no\\
&=\sum _{N=0} ^\infty \sum_{n_1+\dots+n_m=N}\| T_m(\zeta_m,\zeta_m') V_{n_m}(A_m;z_m ,z_m') \cdots T_1(\zeta_1,\zeta_1') V_{n_1}(A_1;z_1 , z_1') \Psi \|\no\\
&\le \sum_{N=0}^\infty \frac{(|z_1-z'_1|+\dots+|z_m-z'_m|)^N}{N!}e^{4mK(E+Nb+ma)}C^{N+m}
\big( E+Nb +(m-1)a+1 \big) ^{1/2} \cdots \big( 
E+(m-1)a+1 \big) ^{1/2}  \norm{ \Psi }.
\end{align}
By d'Alembert's ratio test, the final expression in \eqref{final} converges
locally uniformly in $\zeta_1,\zeta_1',z_1,z'_1,\dots,\zeta_m,\zeta_m',z_m,z'_m$.
\end{proof}
\begin{proof}[Proof of Theorem \ref{time-ordered-exp}]
Let $\im z \le \im z'$. Lemma \ref{boundLem2} \eqref{D-bound} shows that for all $\Psi\in D_\fin$,
\begin{align}\label{U-conv}
U(A;z,z')\Psi&:=\sum_{n=0} ^\infty V_n(A;z,z')\Psi \no\\
&=T\exp\left(-i\int_{\Gamma_{z,z'}} d\zeta A(\zeta)\right)\Psi
\end{align}
exists and is independent of $\Gamma _{z,z'} $. This proves (i).

We prove (ii). Inductively, we see for all integer $ n \ge0 $,
\begin{align}\label{adjoint1} 
V_n (A;z,z') ^* \Psi = V_n (A^*;z' {}^* , z^*) \Psi , \q \Psi \in D_\rm{fin} .
\end{align}
The case $ n=0 $ is trivial. Assume that \eqref{adjoint1} holds for some $ n $. Let $ \Gamma : [0,1] \to \C $ be a continuously differentiable simple curve from $ z' $ to $ z $. Then, we have for all $ \Psi , \Phi \in D _\rm{fin} $,
\begin{align*}
\Expect{ \Psi , V_{n+1} (A;z,z') \Phi  } & = -i \int _{z'} ^z d\zeta \Expect{ \Psi , A(\zeta )V_n (A;\zeta , z') \Phi } \\
& = -i \int _0 ^1 dt \, \Gamma ' (t) \Expect{ \Psi , A(\Gamma(t)) V_n (A;\Gamma (t),z') \Phi } \\
& = \Expect{ i \int _0 ^1 \Gamma ' (t) ^* V_n (A^*;z' {}^* , \Gamma (t) ^*) A^*(\Gamma (t) ^*) \Psi , \Phi } \\
& = \Expect{ i \int _{z' {}^*} ^{z^*} d\zeta \, V_n (A^*;z' {}^* , \zeta ) A^*(\zeta ) \Psi , \Phi } \\
& = \Expect{ V_{n+1} (A^*;z' {}^* , z^*) \Psi , \Phi } ,
\end{align*}
where we have used Lemma \ref{int-rep-and-recursion} (iii) in
the first and the last equality, and the induction hypothesis in the third equality. Thus, \eqref{adjoint1} holds for $ n+1 $, so the induction step is complete. 
Then, by \eqref{adjoint1}, we have for all $ \Psi , \Phi \in D_\rm{fin} $,
\begin{align*}
\Expect{ \Psi , U(A;z,z') \Phi } & = \sum _{n=0} ^\infty \Expect{\Psi , V_n (A;z,z') \Phi } \\
& = \sum _{n=0} ^\infty \Expect{ V_n (A^*;z' {}^* , z^*) \Psi , \Phi } \\
& = \Expect{ U(A^*;z' {}^* , z^*) \Psi , \Phi } .
\end{align*}
This yields the inclusion relation
\begin{align}\label{pre-adjoint}
U (A;z,z') ^* \supset U(A^*;z' {}^* ,z^*) ,
\end{align}
implying that $ U(A;z,z') $ is closable. Therefore, we can take the closure of the both sides of  \eqref{pre-adjoint}, and the desired result follows.
\end{proof}

\begin{Thm}
 \label{sandwich} Let $ T_k ,A_k\, (k= 1, ... , m , \; m \ge 1) $ be $ \mathcal{C}_0 $-class operators. Then, for all $ z_k , z_k ' \in \C \; (k=1,...,m) $ with $\im z_k\le \im z'_k$ and $\zeta_k,\zeta_k'\in \C$, it follows that
\begin{align}
 D_\rm{fin} \subset D(T_m(\zeta_m,\zeta_m') \overline{U(A_m;z_m ,z_m')} \cdots T_1(\zeta_1,\zeta_1') \overline{U(A_1;z_1 , z_1')}) .
  \end{align}
 Moreover, for all $ \Psi \in D_\rm{fin} $,
\begin{align}
&T_m (\zeta_m,\zeta_m')\overline{U(A_m;z_m ,z_m')} \cdots T_1(\zeta_1,\zeta_1') \overline{U(A_1;z_1 , z_1')} \Psi \no\\
\qquad\qquad&= \sum _{ n_1 , ... , n_m  = 0} ^\infty T_m (\zeta_m,\zeta_m')V_{n_m}(A_m;z_m ,z_m') \cdots T_1(\zeta_1,\zeta_1') V_{n_1} (A_1;z_1 , z_1') \Psi , \label{sandwich1}
\end{align}
where the right-hand side converges absolutely, and does not depend upon the summation order.
Furthermore, this convergence is locally uniform in the complex variables $z_1,z_1',\zeta_1,\zeta_1',\dots,z_m,z_m',\zeta_m,\zeta_m'$
\end{Thm}
By Theorem \ref{sandwich}, it is natural to introduce the algebra $\mathfrak{A}$
generated by 
\begin{align}
\left\{T,\overline{U(A;z,z')},e^{i\zeta H_0}\left | T,A\in\mathcal{C}_0,\,z,z',\zeta\in\C,\,\im z\le\im z'\right. \right\}.
\end{align}
It is clear that all $a\in\mathfrak{A}$ is closable since they have densely defined adjoints and the subspace
$D_\fin$ is a common domain of $\mathfrak{A}$. We define a dense subspace $\mathcal{D}$ by
\begin{align}
\mathcal{D}:=\mathfrak{A}D_\fin.
\end{align}
Theorem \ref{sandwich} shows that $\mathcal{D}$ is also a common domain of $\mathfrak{A}$.
Moreover, for all $\Psi\in\mathcal{D}$, there exists a sequence $\{\Psi_N\}_N\subset D_\fin$
such that 
\begin{align}
\Psi_N\to\Psi, \q a\Psi_N\to a\Psi\q(a\in\mathfrak{A})
\end{align}
as $N$ tends to infinity. This implies that if an equality $a=b$ ($a,b\in\mathfrak{A}$) holds 
on $D_\fin$, then $a=b$ on $\mathcal{D}$ and the convergence is locally uniform in 
all the complex variables included in $a$ and $b$.
From this observation, we immediately have
\begin{Cor} Let $A$ be in $\mathcal{C}_0$ class and $z,z'\in\C$ with $\im z\le \im z'$.
Then,
\begin{align}
\mathcal{D} \subset D\left(T\exp\left(-i\int_{\Gamma_{z,z'}} d\zeta A(\zeta)\right)\right)
\end{align}
and for $\Psi\in\mathcal{D}$,
\begin{align}
T\exp\left(-i\int_{\Gamma_{z,z'}} d\zeta A(\zeta)\right)\Psi = \overline{U(A;z,z')}\Psi.
\end{align}
In particular, 
\begin{align}
T\exp\left(-i\int_{\Gamma_{z,z'}} d\zeta A(\zeta)\right)\Psi
\end{align}
is independent of the simple curve $\Gamma_{z,z'}$ and depends only on $z,z'$ if $\Psi\in\mathcal{D}$.

\end{Cor}

\begin{proof}[Proof of Theorem \ref{sandwich}]
We prove the claim by induction on $ m \ge 1 $. 
Let $ m=1 $, and let $ \Psi \in D_\rm{fin} $. By Lemma \ref{boundLem2},
\begin{align}
\sum _{n=0} ^\infty \| T_1(\zeta_1) V_n (A_1;z_1, z_1') \Psi \| < \infty .
\end{align}
Then, since $ T_1(\zeta_1) $ is closed, we get $ U(A_1;z_1, z_1') \Psi \in D(T_1(\zeta_1)) $ and \eqref{sandwich1} for $ m=1 $.

Suppose that the claim is true for some $ m\ge 1 $. Let $ \Psi \in D_\rm{fin} $. By 
Lemma \ref{boundLem2}, one sees
\begin{align}
& \sum _{ n_1 , ... , n_m  = 0} ^\infty \| V_{n_{m+1}}(A_{m+1};z_{m+1} ,z'_{m+1}) \cdots T_1(\zeta_1) V_{n_1}(A_1;z_1 , z_1') \Psi \| < \infty  , \label{sandwich3}\\
& \sum _{ n_1 , ... , n_m  = 0} ^\infty \| T_{m+1}(\zeta_{m+1}) V_{n_{m+1}}(A_{m+1};z_{m+1} ,z'_{m+1}) \cdots T_1(\zeta_1) V_{n_1}(A_1;z_1 , z_1') \Psi \| < \infty .\label{sandwich2} 
\end{align}
Hence, 
we have using induction hypothesis
\begin{align}
T_{m}(\zeta_{m}) \overline{U(A_m;z_{m} ,z'_{m})} \cdots T_1 (\zeta_1)\overline{U(A_1;z_1 , z_1')} \Psi \in D(T_{m+1}
(\zeta_{m+1}) \overline{U(A_{m+1};z_{m+1} , z'_{m+1} )}) 
\end{align}
and \eqref{sandwich1} for $ m+1 $ since $ T_{m+1} $ is closed and $ U(A_{m+1};z_{m+1} ,z_{m+1}') $ are 
closable. 
Thus, the assertion holds also for $ m+1 $.
The local uniformity of the convergence follows the fact that the series in Lemma \ref{boundLem2} \eqref{D-bound}
converges locally uniformly.
\end{proof}

\begin{Thm}\label{diff-eq}
Let $A$ be in $\mathcal{C}_0$ class and $z,z'\in\C$.
\begin{enumerate}[(i)]
\item For all $\Psi\in \mathcal{D}$, the vector valued function
\[ \{(z,z')\,|\,\im z \le \im z'\} \ni(z,z')\mapsto \overline{U(A;z,z')}\Psi  \in\H \]
is analytic on the region $\{\im z < \im z'\}$ and continuous on $\{\im z\le \im z'\}$. 
Moreover, it is a solution of differential equations 
\begin{align}
& \frac{\pa}{\pa z} \overline{U(A;z,z')} \Psi = -i A(z) \, \overline{U(A;z,z')} \Psi , \label{main1-1} \\
& \frac{\pa}{\pa z'} \overline{U(A;z,z')} \Psi = i \overline{U(A;z,z')} A(z') \Psi , \label{main1-2}
\end{align}
on \{$\im z < \im z'$\}. 
\item For all $\Psi\in \mathcal{D}$, the vector valued function 
$\R^2\ni(t,t')\mapsto \overline{U(A;t,t')}\Psi$ is continuously differentiable on the region $\R^2$, satisfying the differential equations 
\begin{align}
& \frac{\pa}{\pa t} \overline{U(A;t,t')} \Psi = -i A(t) \, \overline{U(A;t,t')} \Psi , \label{main1-3} \\
& \frac{\pa}{\pa t'} \overline{U(A;t,t')} \Psi = i \overline{U(A;t,t')} A(t') \Psi . \label{main1-4}
\end{align}
\end{enumerate}
\end{Thm}
\begin{proof}
We prove (i). Since the convergence in \eqref{U-conv} is locally uniform
in $z,z'$ and each $V_n(A;z,z')$ are analytic on all $z,z'\in \C$,
we conclude that $U(A;z,z')$ is analytic on the region $\{\im z < \im z'\}$ and continuous on $\{\im z\le \im z'\}$.
 Due to the
fact that the convergences are uniform in \eqref{D-bound},
one finds
\begin{align}
\sum_{n=0}^\infty A(z)V_n(A;z,z')\Psi &= A(z)U(A;z,z')\Psi,\\
\sum_{n=0}^\infty V_n(A;z,z')A(z')\Psi &= U(A;z,z')A(z')\Psi,
\end{align}
absolutely and locally uniformly in $z,z'$ when $n$ tends to infinity.
By taking $n\to\infty $ in \eqref{inteq1} and \eqref{inteq2}, we obtain
\begin{align}
U(A;z,z')&=1-i\int_{z'}^z d\zeta \,A(\zeta) U(A;\zeta,z'),\label{inteq3} \\
&=1-i\int_{z'}^z d\zeta \,U(A;z,\zeta)A(\zeta), \label{inteq4}
\end{align}
on $D_\fin$. By the remark just below the statement of Theorem \ref{sandwich}, 
integral equations \eqref{inteq3} and \eqref{inteq4} can be extended to $\mathcal{D}$ in the form
\begin{align}
\overline{U(A;z,z')}&=1-i\int_{z'}^z d\zeta \,A(\zeta) \overline{U(A;\zeta,z')},\label{inteq3'} \\
&=1-i\int_{z'}^z d\zeta \,\overline{U(A;z,\zeta)}A(\zeta). \label{inteq4'}
\end{align}
Differentiating these expression with respect to $z$ or $z'$,
one finds \eqref{main1-1} and \eqref{main1-2}.

Considering the case where $z,z'$ are real, we obtain (ii) in the same manner.

\end{proof}

\begin{Thm}\label{properties} Let $A\in\mathcal{C}_0$ and $ z,z',z'' \in \C $. Then, 
the following properties hold.
\begin{enumerate}[(i)]
\item If $\im z\le\im z'\le \im z''$, the equalities
\begin{align}
\overline{U(A;z,z)} = I, \q  \overline{U(A;z,z')}\; \overline{U(A;z' , z'')} = \overline{U(A;z,z'') }
\end{align}
hold on the subspace $\mathcal{D}$, where $I$ is the identity operator.

\item Let $\im z\le \im z'$. Then, $U(A;z,z')$ is translationally invariant in the sense that the equality
\begin{align}\label{parallel3}
e^{izH_0} \overline{U(A;z',z'')} e^{-izH_0} \Psi = \overline{U(A;z' +z , z'' +z)} 
\end{align}
holds on the subspace $\mathcal{D}$.
\item For all $ t,t' \in \R $, $ \overline{U(A; t,t')} $ is unitary. Moreover, for all $ t,t',t'' \in \R $, the operator equality
\begin{align}\label{UUU}
\overline{U(A;t,t')} \; \overline{U(A;t',t'')} = \overline{U(A;t,t'')}  
\end{align}
holds.
\end{enumerate}
\end{Thm}
\begin{proof}
\begin{enumerate}[(i)]
\item 
Fix $z,z''$ so that $\im z<\im z''$. Then,
 by Theorem \ref{time-ordered-exp}, for all $ \Psi, \Phi \in \mathcal{D} $ and $ z' \in \C $
 with $\im z'\in (\im z,\im z'')$,
\begin{align*}
& \frac{d}{dz'} \Expect{ \Phi , \overline{U(A;z,z')}\; \overline{U(A;z' , z'')} \Psi } \\
&= \frac{d}{dz'} \Expect{ \overline{U(A^*;z' {}^* ,z ^*)} \Phi , \overline{U(A;z' , z'')} \Psi } \\
&=  \Expect{ -i A^*(z' {}^* ) \, \overline{U(A^*;z' {}^* , z^*) }\Phi , \overline{U(A;z' , z'')} \Psi } +
 \Expect{ \overline{U(A;z' {}^* , z^*)} \Phi , -i A(z') \overline{U(A;z' , z'') }\Psi } \\
&=  0 .
\end{align*}
This yields that 
\begin{align}\label{pre-gr}
z'\to \Expect{ \Phi , \overline{U(A;z,z')}\; \overline{U(A;z' , z'')} \Psi }
\end{align}
is constant on the region $\{z'\,|\, \im z'\in (\im z,\im z'')\}$. But this function is continuous 
on its closure, implying that it must be constant on the closed region $\im z\le\im z' \le \im z''$.
Taking $z'=z$ we have 
\begin{align}\label{const}
 \Expect{ \Phi , \overline{U(A;z,z')}\; \overline{U(A;z' , z'')} \Psi }=
 \Expect{\Phi ,\overline{U(A;z , z'')} \Psi}
 \end{align}
for all $\im z\le\im z' \le \im z''$ with $\im z<\im z''$. 
Fix $z,z'\in\C$ so that $\im z=\im z'$ and regard both sides of \eqref{const} as a function of $z''$.
Since these functions are continuous on $\{z''\,|\, \im z\le \im z''\}$ and coincide on $\{z''\,|\, \im z < \im z''\}$,
they must coincide on $\{z''\,|\, \im z\le \im z''\}$. This completes the proof.
\item We first show by induction on $ n \ge 0$ that
\begin{align}\label{parallel1}
e^{izH_0} V_n (A;z' , z'') e^{-izH_0} \Psi = V_n (A;z' +z , z'' +z)  \Psi , \q \Psi \in D _\rm{fin} .
\end{align}
The case $ n=0 $ is trivial. Assume that \eqref{parallel1} holds for some $ n $. Then, we have for all $ \Psi \in D _\rm{fin} $,
\begin{align*}
e^{izH_0} V_{n+1} (A;z',z'') e^{-izH_0} \Psi & = -i \int _{z''} ^{z'} d\zeta e^{izH_0 } A(\zeta ) V_n (A;\zeta , z'') e^{-izH_0} \Psi \\
& = -i \int _{z''} ^{z'} d\zeta\, A(\zeta + z) V_n (A;\zeta + z , z'' +z) \Psi \\
& = -i \int _{z''+z} ^{z'+z} d\zeta \,A(\zeta ) V_n (A;\zeta , z'' +z) \Psi \\
& = V_{n+1} (A;z' +z , z'' +z) \Psi ,
\end{align*}
where we have used the basic property $ e^{izH_0} A(\zeta ) e^{-izH_0} \Phi = A(\zeta + z) \Phi $ ($\Phi \in D_\rm{fin}$)  
in the second equality and the induction hypothesis in the third. This completes the induction. 

Summing up the both sides of \eqref{parallel1} over all $ n \ge 0 $, and using the closedness of $ e^{izH_0} $, we obtain 
\begin{align}
e^{izH_0}U(A;z',z'')e^{-izH_0}=U(A;z'+z,z''+z)
\end{align}
on $D_\fin$. But both sides belong to $\mathfrak{A}$, this equality holds on $\mathcal{D}$ in the form
\begin{align}
e^{izH_0}\overline{U(A;z',z'')}e^{-izH_0}=\overline{U(A;z'+z,z''+z)}.
\end{align}

\item Similar to the proof of \cite[Theorem 2.4]{FutakuchiUsui2013}.
\end{enumerate}
\end{proof}

\begin{Thm}\label{Time-order-ext}
Let $A_1,\dots A_k,B\in\mathcal{C}_0$, and $z,z'\in\C$ with $\im z\le\im z'$. 
Let $\Gamma_{z,z'}$ be a simple curve from $z'$ to $z$ and $\zeta_1,\dots,\zeta_k\in\Gamma$ be different from
each other.
Then, we have
\begin{align}
\mathcal{D} \subset D\left(TA_1(\zeta_1)\dots A_k(\zeta_k)\exp \left(-i\int_{\Gamma_{z,z'} }d\zeta \,B(\zeta) \right)\right)
\end{align}
and 
\begin{align}
&TA_1(\zeta_1)\dots A_k(\zeta_k)\exp \left(-i\int_{\Gamma_{z,z'}} d\zeta \,B(\zeta) \right)\Psi\no\\
&\qquad\qquad=\overline{U(B;z,\zeta_{j_1})}A_{j_1}(\zeta_{j_1}) \overline{U(B;\zeta_{j_1},\zeta_{j_2})}\dots \overline{U(B;\zeta_{k-1},\zeta_k)}
A_{j_k}(\zeta_{j_k}) \overline{U(B;\zeta_{j_k},z')}\Psi
\end{align}
for all $\Psi\in \mathcal{D}$, where $(j_1,\dots,j_k)$ is the permutation of $(1,2,\dots,k)$ with
$\zeta_{j_1}\succ \dots \succ\zeta_{j_k}$.
\end{Thm}

\begin{proof}
Put
 \begin{align}
 A_{k+1}=\dots=A_{k+n}=B.
 \end{align}
 We can assume that 
 \begin{align}
 \zeta_1\succ \dots \succ \zeta_{k}
 \end{align}
 without loss of generality.
 Take $\Psi\in \mathcal{D}$. For all $n\in\N$ and all 
$\sigma\in\mathfrak{S}_{k+n}$,
it is clear that the mapping
\begin{align}
(\zeta_{k+1},\dots,\zeta_{k+n})\mapsto A_{\sigma(k)}(\zeta_{\sigma(k)})\dots A_{\sigma(k+n)}(\zeta_{\sigma(k+n)})\Psi
\end{align}
is analytic and thus the strong integral
\begin{align}\label{s-int}
&\frac{(-i)^n}{n!} \int_{\Gamma_{z,z'}^n} d\zeta_{k+1}\dots d\zeta_{k+n}\, TA_1(\zeta_1)\dots A_{k}(\zeta_{k})B(\zeta_{k+1})\dots B(\zeta_{k+n}) \Psi \no\\
&\q\q\q= \frac{(-i)^n }{n!}\sum_{\sigma\in\mathfrak{S}_{k+n}}\int_{P'_{n,\sigma}} d\zeta_{k+1}\dots d\zeta_{k+n}\, A_{\sigma(1)}(\zeta_{\sigma(1)})\dots A_{\sigma(k+n)}(\zeta_{\sigma(k+n)}) \Psi
\end{align}
exists. The integral on the right hand side vanishes unless $\sigma$ is of the following form:
There are $l_1,\dots,l_{k+1}$ satisfying
\begin{align}\label{l-condition}
l_1,\dots,l_{k+1}\ge 0,\q l_1+\dots + l_{k+1}=n
\end{align}
and 
\begin{align}
\sigma(l_1+1)=1,\q\sigma(l_1+l_2+2)=2,\q\dots,\q\sigma(l_1+\dots+l_{k}+k)=k.
\end{align}
If we denote such permutation $\sigma$ by $\sigma_{l_1,\dots,l_{k+1}}$, the summation over $\sigma$ can be performed
by summing up all $\sigma$'s of the form $\sigma=\sigma_{l_1,\dots,l_{k+1}}$ for some $l_1,\dots,l_{k+1}$ (there are $n!$ such $\sigma$'s
for each fixed $l_1,\dots,l_{k+1}$)
, and then summing over all $l_1,\dots,l_{k+1}$
satisfying \eqref{l-condition}:
\begin{align}
\sum_{\sigma\in\mathfrak{S}_{k+n}}=\sum_{\substack{l_1,\dots,l_{k+1}\ge 0\\l_1+\dots + l_{k+1}=n}}\sum_{\sigma=\sigma_{l_1,\dots,l_{k+1}}}.
\end{align}
The integration in \eqref{s-int} depends only upon $l_1,\dots,l_{k+1}$, but not upon the concrete form of $\sigma=\sigma_{l_1,\dots,l_{k+1}}$,
and thus the summation over $\sigma=\sigma_{l_1,\dots,l_{k+1}}$ just gives the factor $n!$.  Then, we have
\begin{align}\label{comp-sigma}
& \frac{(-i)^n }{n!}\sum_{\sigma\in\mathfrak{S}_{k+n}}\int_{P'_{n,\sigma}} d\zeta_{k+1}\dots d\zeta_{k+n}\, A_{\sigma(1)}(\zeta_{\sigma(1)})\dots A_{\sigma(k+n)}(\zeta_{\sigma(k+n)}) \Psi\no\\
&\q\q= \frac{(-i)^n }{n!}
\sum_{\substack{l_1,\dots,l_{k+1}\ge 0\\l_1+\dots + l_{k+1}=n}}\sum_{\sigma=\sigma_{l_1,\dots,l_{k+1}}}
\int_{z\succ \tau^{(1)}_1\succ \dots \succ \tau^{(1)}_{l_1}\succ \zeta_1\succ \dots \succ \zeta_k\succ \tau^{(k+1)}_{1}\succ\dots\succ \tau^{(k+1)}_{l_{k+1}}\succ z'}  d\tau^{(1)}_{1}\dots d\tau^{(1)}_{l_1}\dots  d\tau^{(k+1)}_{1}\dots d\tau^{(k+1)}_{l_{k+1}}\, \no\\
&\q\q\q\q\q\q\q\q B(\tau^{(1)}_1)\dots B(\tau^{(1)}_{l_1}) A_1({\zeta_1}) 
\dots A_k(\zeta_k)B(\tau^{(k+1)}_1)\dots B(\tau^{(k+1)}_{l_{k+1}}) \Psi \no\\
&\q\q=\sum_{\substack{l_1,\dots,l_{k+1}\ge 0\\l_1+\dots + l_{k+1}=n}}\left(\frac{(-i)^{l_1}}{l_1!}\int_{\Gamma_{z,\zeta_1}^{l_1}} d\tau^{(1)}_{1}\dots d\tau^{(1)}_{l_1}\, TB(\tau^{(1)}_1)\dots B(\tau^{(1)}_{l_1})\right)
A_1(\zeta_1)\dots \no\\
&\q\q\q\q\q\q\q\q \dots A_k(\zeta_k)\left(\frac{(-i)^{l_{k+1}}}{l_{k+1}!}\int_{\Gamma_{\zeta_k,z'}^{l_{k+1}}} d\tau^{(k+1)}_{1}\dots d\tau^{(k+1)}_{l_{k+1}}\, TB(\tau^{(k+1)}_1)\dots B(\tau^{(k+1)}_{l_{k+1}})\right)\Psi\no\\
&\q\q=\sum_{\substack{l_1,\dots,l_{k+1}\ge 0\\l_1+\dots + l_{k+1}=n}}V_{l_1}(B;z,\zeta_1)A_1(\zeta_1)\dots A_k(\zeta_k)V_{l_{k+1}}(B;\zeta_k,z')\Psi.
\end{align}
The final expression in \eqref{comp-sigma} is absolutely summable with respect to $n=0,1,2,\dots$ to give
\begin{align}
\overline{U(B;z,\zeta_1)}A_1(\zeta_1)\dots A_k(\zeta_k)\overline{U(B;\zeta_k,z')}\Psi
\end{align}
by Theorem \ref{sandwich},
which means that $\Psi$ belongs to the subspace
\begin{align}
D\left(TA_1(\zeta_1)\dots A_k(\zeta_k)\exp\left(-i\int_{\Gamma_{z,z'}}d\zeta\, B(\zeta)\right)\right),
\end{align}
and 
\begin{align}
TA_1(\zeta_1)\dots A_k(\zeta_k)\exp\left(-i\int_{\Gamma_{z,z'}}d\zeta\,B(\zeta) \right)\Psi 
= \overline{U(B;z,\zeta_1)}A_1(\zeta_1)\dots A_k(\zeta_k)\overline{U(B;\zeta_k,z')}\Psi.
\end{align}
This completes the proof.
\end{proof}
\section{Complex time evolution and Gell-Man -- Low formula}
In this section, we consider the operator
\begin{align}
H=H_0+H_1
\end{align}
with $H_1\in\mathcal{C}_0$, and we state and derive the Gell-Mann -- Low formula.
 In what follows, we shortly denote 
\begin{align}
V_n(z,z'):=V_n(H_1;z,z'),\q U(z,z'):=U(H_1;z,z').
\end{align}
We define complex time evolution operator 
\begin{align}
W(z):=e^{-izH_0}\overline{U(z,0)}
\end{align}
for $z\in\C$ with $\im z\le 0$. The operator $W(z)$ generates the ``complex time evolution" in the 
following sense:
\begin{Thm}\label{Scheq}
For all $\Psi\in\mathcal{D}$, the mapping $z\mapsto W(z)\Psi$ is analytic 
on the lower half plain and satisfies the ``complex Schr\"odinger equation" 
\begin{align}
\frac{d}{dz}W(z)\Psi = -iHW(z)\Psi.
\end{align}
\end{Thm}
\begin{proof}
We first remark that $\mathcal{D}\subset D(H_0)$. This can be seen by noting that
$\mathcal{D}\subset D(e^{H_0})\subset D(H_0)$.
By Theorem \ref{time-ordered-exp}, one can easily estimate
\begin{align}
\norm{\frac{W(z+h)\Psi-W(z)\Psi}{h} - (-iH)W(z)\Psi}
\end{align}
to know that this vanishes in the limit $h\to 0$.

\end{proof}

\begin{Thm}\label{explicitThm}
Suppose that $H_1$ is $ \mathcal{C}_0 $-class symmetric operator. Then, $H$ is self-adjoint and bounded below.
Moreover, it follows that 
\begin{align}
\overline{W(z)} \Psi = e^{-izH},
\end{align}
for all $z\in\C$ with $\im z \le 0$.
In particular, it follows that
\begin{align}\label{explicit100}
\overline{U(z,z')} = e^{izH_0} e^{-i(z-z')H} e^{-iz'H_0} , \q \im z\le \im z' .
\end{align}
\end{Thm}
\begin{proof}
From Assumption \ref{Ass3}, $H_1$ is infinitesimal with respect to $H_0$ and thus $H$ is self-adjoint with
$D(H)=D(H_0)$, and bounded below by the Kato-Rellich Theorem. 

By Theorem \ref{Scheq}, we can differentiate for all $\Psi\in\mathcal{D}$, $\Phi\in D_0(H):=\cup_{L\in\Real} R (E_H([-L,L]))$,
and $z\in \C$ with $\im z <0$,
\begin{align}
\frac{d}{dz}\Expect{e^{-iz^*H}\Phi,W(z)\Psi}&=\Expect{-iHe^{-iz^*H}\Phi,W(z)\Psi} + \Expect{e^{-iz^*H}\Phi,-iW(z)\Psi}\no\\
&=0.
\end{align}
Thus, one finds
\begin{align}\label{above}
\Expect{\Phi,\Psi}=\Expect{e^{-iz^*H}\Phi,W(z)\Psi},
\end{align}
for all $\Psi\in\mathcal{D}$ and $\Phi\in D_0(H)$.
Since $D_0(H)$ is a core of $e^{-iz^*H}$, we obtain from \eqref{above} $W(z)\Psi\in D(e^{izH})$ and
\begin{align}
e^{izH}W(z)\Psi=\Psi.
\end{align}
Hence, we arrive at
\begin{align}\label{con-half}
W(z)\Psi = e^{-izH}\Psi,
\end{align}
for all $z\in\C$ with $\im z < 0$. But since both sides of \eqref{con-half} are continuous on the region $\im z \le 0$,
\eqref{con-half} must hold on $\im z \le 0$. Since the both sides are bounded, one has
\begin{align}
\overline{W(z)} = e^{-izH},\q \im z \le 0.
\end{align}

For $z,z'$ satisfying $\im z \le \im z'$, we have from \eqref{parallel3}
\begin{align}
W(z-z')\Psi &= e^{-i(z-z')H_0}\overline{U(z-z',0)}\Psi \no\\
&=e^{-izH_0}\overline{U(z,z')}\,e^{iz'H_0}\Psi,\q \Psi\in\mathcal{D}.
\end{align} 
This implies
\begin{align}
\overline{U(z,z')}\Psi = e^{izH_0}e^{-i(z-z')H}e^{iz'H_0}\Psi.
\end{align}
If $z,z'$ are real, the right-hand-side is unitary, and thus the last assertion follows. 
\end{proof}

We introduce the assumptions needed to derive the Gell-Mann -- Low formula. 
For a linear operator $ T $, we denote the spectrum of $ T $ by $ \sigma (T) $. If $ T $ is self-adjoint and bounded from below, then we define
\begin{align}
E_0 (T) := \inf \sigma (T) .
\end{align}
We say that $ T $ has a ground state if $ E_0 (T) $ is an eigenvalue of $ T $. In that case, $ E_0 (T) $ is called the \textit{ground energy} of $T$, and each non-zero vector in $ \ker (T-E_0 (T)) $ is called a \textit{ground state} of $ T $. If $ \dim \ker (T-E_0 (T)) =1 $, we say that $ T $ has a unique ground state.
The following assumption are used to prove the Gell-Mann -- Low formula.

\begin{Ass}\label{Ass3} 
\begin{enumerate}[(I)]
\item $ H_0 $ has a unique ground state $ \Omega _0 \; (\| \Omega _0 \| =1 ) $, and the ground energy is zero: $ E_0 (H_0) =0 $.

\item $ H $ has a unique ground state $ \Omega \; (\| \Omega \| =1 )$. 

\item $ \Expect{\Omega , \Omega _0} \neq 0 $.
\end{enumerate}
\end{Ass}

Under Assumption \ref{Ass3}, we define $m$-point Green's function $G_m(z_1,\dots,z_m)$ by
\begin{align}
G_m(z_1,\dots,z_m) := e^{i(z_1-z_m)E_0(H)}\Expect{\Omega, A_1W(z_1-z_2)A_2\dots A_{m-1}W(z_{m-1}-z_m)\Omega},
\end{align}
for $\im z_1\le\dots\le\im z_m$ whenever the right-hand-side is well-defined.
The Gell-Mann and -- Low formula is given by:

\begin{Thm}\label{GML} 
Suppose that Assumption \ref{Ass3} holds and $H_1$ is symmetric. Let $ A_k \; (k=1,..., m , \; m\ge 1) $ be linear operators having the following properties:
\begin{enumerate}[(I)]
\item Each $ A_k $ is in $ \mathcal{C}_0 $-class. 

\item For each $ k $, there exist integer $ r_k \ge 0 $ such that, for all $ n\in \N $, $ A_k $ maps $D(H^{n+r_k}) $ into $ D(H^n) $. 
\end{enumerate}
Let $ z_1 ,..., z_m \in \C$ with $\im z_1\le \dots \le \im z_m$. Choose a simple curve $\Gamma_T^\e$ from $-T(1-i\e)$ to $T(1-i\e)$ ($T,\e>0$) on which
$z_1\succ \dots \succ z_m$. Then, $m$-point Green's function $G_m(z_1,\dots,z_m)$ is well-defined and satisfies the formula
\begin{align}
G_m(z_1,\dots,z_m)
 = \lim _{ T \to \infty } \frac{ \Expect{ \Omega _0 , TA_1(z_1)\dots A_m(z_m)
\exp \left(-i\int_{\Gamma_T^\e}d\zeta\, H_1(\zeta)\right) \Omega _0 } }{\Expect{\Omega _0 , 
T\exp \left(-i\int_{\Gamma_T^\e}d\zeta\, H_1(\zeta)\right)\Omega _0}  }. \label{GMLeq}
\end{align}
\end{Thm}

To prove the Gell-Mann -- Low formula \eqref{GMLeq}, we prepare some lemmas.
We denote $E_0(H)$ simply by $E_0$.
\begin{Lem}\label{adiabatic}
For $\e>0$ and Borel function $f:\R\to\C$, we have
\begin{align}
\lim_{T\to\infty} f(H) e^{iT(\pm1-i\e)E_0}W(T(\pm1-i\e))\Psi = f(E_0) P_0\Psi, \q \Psi \in D(f(H)),
\end{align}
where $P_0$ is the Projection onto the closed subspace $\ker (H-E_0)$. 
\end{Lem} 
\begin{proof}
By the functional calculus and Lebesgue's convergence Theorem, we have
\begin{align}
\norm{f(H)e^{iT(\pm1-i\e)E_0}W(T(\pm1-i\e))\Psi - f(E_0)P_0\Psi}^2&=\norm{f(H)e^{\mp iT(H-E_0)}e^{-T\e (H-E_0)}\Psi -f(E_0) E_H(\{E_0\})\Psi}^2\no\\
&=\int_{[E_0,\infty)}d\norm{E_H(\lambda)\Psi}^2 \left| f(\lambda)(e^{-T\e (\lambda -E_0)}\Psi -\chi_{\{E_0\}}(\lambda))\right|^2\no\\
&=\int_{(E_0,\infty)}d\norm{E_H(\lambda)\Psi}^2 \left|f(\lambda) e^{-T\e (\lambda -E_0)}\Psi\right|^2\no\\
&\to 0,
\end{align}
as $T$ tends to infinity.
\end{proof}
\begin{Lem}\label{bdd-ak}
Under the same assumptions of Theorem \ref{GML}, the operators 
\begin{align}
\widetilde{A_k}  :=  (H-\zeta )^{\sum _{j=1} ^{k-1} r_j } A_k (H-\zeta ) ^{-\sum _{j=1}^k r_j } , \q k=1,...,m,
\end{align}
are bounded.
\end{Lem}
\begin{proof}
From the assumptions, 
\begin{align}
 A_k (H-\zeta ) ^{-\sum _{j=1}^k r_j }\Psi\in D(H^{\sum _{j=1} ^{k-1} r_j }),
\end{align}
for all $\Psi\in\H$. Thus, 
\[ D(\widetilde{A_k})=\H .\]

On the other hand, it is easy to check that $\widetilde{A_k}$'s are closed.
Hence, by the closed graph theorem, each $\widetilde{A_k}$'s are bounded. 
\end{proof}
\begin{Lem}\label{tea}
Under the same assumptions of Theorem \ref{GML}, it follows that
\begin{align}
&\lim_{T\to\infty} A_1W(z_1-z_2)A_2\dots A_{m-1}W(z_{m-1}-z_m)A_m f(H) e^{iT(\pm1-i\e)}W(T(\pm1-i\e))\Psi \no\\
&\q\q\q\q =  
A_1W(z_1-z_2)A_2\dots A_{m-1}W(z_{m-1}-z_m)A_m f(E_0)P_0\Psi, \q \Psi\in \bigcap_{n\in\N}D(H^nf(H)).
\end{align}
for all Borel function $f:\R\to\C$.
\end{Lem}

\begin{proof} Under the present assumptions, we see that each $ A_k $ leaves the subspace $ \bigcap _{n=1} ^\infty D(H^n) $ invariant, and thus 
\begin{align}
 \Psi \in D\left(A_1W(z_1-z_2)A_2\dots A_{m-1}W(z_{m-1}-z_m)A_m f(H) e^{iT(\pm1-i\e)}W(T(\pm1-i\e))\right) .
 \end{align} 
 Now let $ \zeta \in \C \backslash \R $. Then, we can rewrite 
\begin{align}\label{morning}
A_1W(z_1-z_2)A_2\dots A_{m-1}W(z_{m-1}-z_m)A_m
= \widetilde{A_1} W(z_1-z_2) \cdots \widetilde{A_{m}} W(z_{m-1}-z_m) (H-\zeta ) ^{\sum _k r_k }  
\end{align}
with 
\begin{align}
\widetilde{A_k}  :=  (H-\zeta )^{\sum _{j=1} ^{k-1} r_j } A_k (H-\zeta ) ^{-\sum _{j=1}^k r_j } , \q k=1,...,m.
\end{align}
Note that each of $ \widetilde{A_k} $'s and $W(z_{k-1}-z_k)$'s is a bounded operator by 
Theorem \ref{explicitThm} and Lemma \ref{bdd-ak}.
Then, by Lemma \ref{adiabatic}, one sees that for all $ n \ge 1 $,
\begin{align}\label{adiabatic-n}
\lim_{T\to\infty}(H- \zeta ) ^n e^{iT(\pm1-i\e)}W(T(\pm1-i\e)) \Psi= (E_0-\zeta)^n P_0\Psi=(H-\zeta)^n P_0\Psi,
\end{align}
which implies the desired result.
\end{proof}

\begin{proof}[Proof of Theorem \ref{GML}]
Put 
\begin{align}
\mathcal{O}(z_1,\dots,z_m):=A_1W(z_1-z_2)A_2\dots A_{m-1}W(z_{m-1}-z_m)A_m.
\end{align}
From Assumption \ref{Ass3}, one finds
\begin{align}
\Omega = \frac{P_0\Omega_0}{\norm{P_0\Omega_0}},
\end{align}
to obtain
\begin{align}\label{p-GML}
G_m(z_1,\dots,z_m)=e^{i(z_1-z_m)E_0}\frac{\Expect{P_0\Omega_0,  \mathcal{O}(z_1,\dots,z_m) P_0 \Omega_0}}{\Expect{P_0\Omega_0,
P_0\Omega_0}}.
\end{align} 
By Lemmas \ref{adiabatic} and \ref{tea}, we have
\begin{align}\label{was}
\frac{\Expect{P_0\Omega_0,  \mathcal{O}(z_1,\dots,z_m) P_0 \Omega_0}}{\Expect{P_0\Omega_0,
P_0\Omega_0}}&=
\lim_{T\to\infty}\frac{\Expect{e^{-iz_1^*(H-E_0)}W(T(-1-i\e))\Omega_0,  \mathcal{O}(z_1,\dots,z_m) e^{-iz_m(H-E_0)}W(T(1-i\e)) \Omega_0}}{\Expect{W(T(-1-i\e))\Omega_0,
W(T(1-i\e))\Omega_0}}.
\end{align}
Using Theorem \ref{explicitThm}, we find
\begin{align}
e^{-iz_1^*(H-E_0)}W(T(-1-i\e))&=e^{iz_1^*E_0}e^{-iz_1^*H_0}\overline{U(z_1^*,T(1+i\e))}e^{iT(1+i\e)H_0} \\
e^{-iz_m(H-E_0)}W(T(1-i\e))&=e^{iz_mE_0}e^{-iz_mH_0}\overline{U(z_m,-T(1-i\e))}e^{-iT(1-i\e)H_0}
\end{align}
on $\mathcal{D}$.
Therefore, by Theorem \ref{Time-order-ext} the numerator on the right-hand-side of  \eqref{was} can be rewritten as
\begin{align}\label{nu}
&e^{-i(z_1-z_m)E_0} \Expect{\Omega_0,\overline{U(T(1-i\e),z_1)}A_1(z_1)\overline{U(z_1,z_2)}\dots\overline{U(z_{m-1},z_m)}
A_m(z_m)\overline{U(z_m,-T(1-i\e))}\Omega_0}\no\\
&\q\q\q = e^{-i(z_1-z_m)E_0}\Expect{\Omega_0,TA_1(z_1)\dots A_m(z_m)\exp\left(-i\int_{\Gamma^\e_T}d\zeta\,H_1(\zeta)\right)\Omega_0}
\end{align}
and the denominater as
\begin{align}\label{de}
\Expect{\Omega_0, U(T(1-i\e),-T(1-i\e))\Omega_0}=\Expect{\Omega_0,T\exp\left(-i\int_{\Gamma^\e_T}d\zeta\,H_1(\zeta)\right)\Omega_0}.
\end{align}
Finally, inserting \eqref{was}, \eqref{nu}, and \eqref{de} into \eqref{p-GML}, we arrive at the 
Gell-Mann -- Low formula \eqref{GMLeq}.
\end{proof}

\section{Application to QED}\label{QEDsection}

In this section we apply the abstract results obtained in the preceding sections to QED with several regularizations in the Coulomb gauge. Our main goal here is Theorem \ref{GML-QED}, which shows that QED with regularizations satisfies Gell-Mann -- Low formula. To prove this, it is sufficient to see that the conditions of Theorem \ref{GML} hold. Under suitable hypotheses, it is not difficult to prove that the interaction Hamiltonian and each field operator are in $ \mathcal{C}_0 $-class (Lemmas \ref{field-spec} and \ref{QEDass1}). However, the existence of the ground state (Assumption \ref{Ass3}) and the condition (II) of Theorem \ref{GML} are not obvious at all. The existence of the ground state is discussed in \cite{MR2541206}. To check the condition (II), we need some preliminaries (Lemmas \ref{prelimLem1}-\ref{preD2}).


\subsection{Fock spaces}\label{subsection1}

Let $ \H $ be a complex separable Hilbert space, and $ \ot ^n \H \; (n \in \N ) $ the $ n $-fold tensor product of $ \H $. Let $ \mathfrak{S} _n $ be the symmetric group of order $ n $ and $ U_\sigma \; (\sigma \in \mathfrak{S} _n) $ be a unitary operator on $ \ot ^n \H $ such that
\begin{align}
U_\sigma (\psi _1 \otimes \cdots \otimes \psi _n ) = \psi _{\sigma (1)} \otimes \cdots \otimes \psi _{\sigma (n)} , \q \psi _j \in \H , \q j=1,..., n .
\end{align}
Then, the symmetrization operator $ S_n $ and the anti-symmetrization operator $ A_n $ are defined by
\begin{align}\label{symmetrization}
S_n := \frac{1}{n!} \sum _{\sigma \in \mathfrak{S}_n } U_\sigma , \q A_n := \frac{1}{n!} \sum _{\sigma \in \mathfrak{S}_n } \sgn (\sigma ) U_\sigma ,
\end{align}
where $ \sgn (\sigma ) $ is the signature of the permutation $ \sigma $. The operators $ S_n $ and $ A_n $ are orthogonal projections on $ \ot ^n \H $. Hence, the subspaces
\begin{align}
\ot _\rm{s} ^n \H := S_n \big( \ot ^n \H \big) , \q \wg ^n \H := A_n \big( \ot ^n \H \big) 
\end{align}
are Hilbert spaces. 
We set $ \ot _\rm{s} ^0 \H := \C , \; \wg ^0 \H := \C $, and define
\begin{align}
\F _\bb (\H ) := \op _{n=0} ^\infty \ot _\rm{s} ^n \H , \q \F _\ff (\H ) := \op _{n=0} ^\infty \wg ^n \H .
\end{align}
$ \F _\bb (\H ) $ (resp. $ \F _\ff (\H ) $) is called the Boson (resp. Fermion) Fock space over $ \H $.

\subsection{Second quantization operators}

For a densely defined closable operator $ T $ on $ \H $ and $ j= 1,...,n $, we define a linear operator $ \widetilde{T}_j  $ on $ \ot ^n \H $ by
\begin{align}
\widetilde{T}_j := I \otimes \cdots \otimes I \otimes \stackrel{j\text{-th}}{T} \otimes I \otimes \cdots \otimes I .
\end{align}
For each integer $ n \ge 0 $, we define a linear operator $ T^{(n)} $ on $ \ot ^n $ by
\begin{align}
T^{(0)} := 0 , \q T^{(n)} := \overline{\sum _{j=1} ^n \widetilde{T}_j \upharpoonright \hot ^n D(T)} , \q n \ge 1,
\end{align}
where $ \hot ^n D(T) $ denotes the $ n $-fold algebraic tensor product of $ D(T) $. Then, the infinite direct sum operator
\begin{align} 
\dd \Ga (T) := \op _{n=0} ^\infty T ^{(n)} 
\end{align}
on $ \F (\H ) $ is called the \textit{second quantization of $ T $}. If $ T $ is non-negative self-adjoint, then so is $ \dd \Ga (T) $. It is easy to see that $ T^{(n)} $ is reduced by $ \ot _\rm{s} ^n \H $ and $ \wg ^n \H $ respectively. We denote the reduced part of $ T^{(n)} $ to $ \ot _\rm{s} ^n \H $ and $ \wg ^n \H $ by $ T^{(n)} _\bb $ and $ T^{(n)}_\ff $ respectively. We set 
\begin{align} 
\dd \Ga _\bb (T) := \op _{n=0} ^\infty T ^{(n)} _\bb , \q \dd \Ga _\ff (T) := \op _{n=0} ^\infty T ^{(n)} _\ff .
\end{align}
The operator $ \dd \Ga _\bb (T) $ (resp. $ \dd \Ga _\ff (T) $) is called the boson (resp. fermion) second quantization operator. 

For a densely defined closable operator $ T $ on $ \H $, we define a linear operator $ \Ga (T)  $ on $ \F ( \H ) $ by
\begin{align} 
\Ga (T) := \op _{n=0} ^\infty \big( \ot ^n T \big) .
\end{align}
We denote the reduced part of $ \Ga (T) $ to $ \F _\bb (\H ) $ and $ \F _\ff (\H ) $ by $ \Ga _\bb (T) $ and $ \Ga _\ff (T) $ respectively.

\begin{Lem}\label{2-33} Let $ K_j \; (j=1,..., n , \; n \ge 1) $ be strongly commuting self-adjoint operators on a Hilbert space $ \H $, and let $ E^n := E_{K_1} \times \cdots \times E_{K_n} $ be a product measure. Set
\begin{align}
P(J) := E^n \Big( \big\{ \la = (\la _1 , ... , \la _n ) \in \R ^n \; \Big| \; \sum _{j=1} ^n \la _j \in J  \big\} \Big) , \q J \in \B ^1 .
\end{align}
Then, $ \{ P(J) \; | \; J \in \B ^1 \} $ is the spectral measure of a self-adjoint operator $ \overline{\sum _{j=1} ^n K_j} $,
where $\B^1$ denotes the set of all the Borel measurable sets in $\R$.
\end{Lem}

\begin{proof} See e.g., \cite[Lemma 2-33]{AraiFock}.
\end{proof}

\begin{Lem}\label{spec-sum} Let $ T $ be a self-adjoint operator in a separable Hilbert space $ \H $. Then, the following (i) and (ii) hold.
\begin{enumerate}[(i)]
\item Let $ E_T^n := E_{\widetilde{T}_1} \times \cdots \times E_{\widetilde{T}_n} $ be a product measure. Then,
\begin{align}
E_{T^{(n)}} (J) = E_T ^n \big( \big\{ (\la _1 , ... , \la _n ) \in \R ^n \; \big| \; \sum _{j=1} ^n \la _j \in J \big\} \big) , \q J \in \B ^1 . 
\end{align}

\item For all $ B_1 , B_2 \in \B ^1 $ and $ n \ge 0 $, it follows that
\begin{align}
R(E_T (B_1)) \otimes R(E_{T^{(n)}} (B_2)) \subset R(E_{T^{(n+1)}}( B_1 + B_2 ) ) \q \text{on} \q \H \otimes \big( \ot ^n \H \big) ,
\end{align}
where $ B_1 + B_2 := \{ \la _1 + \la _2 \in \R \; | \; \la _j \in B_j , \; j=1,2 \} $.
\end{enumerate}
\end{Lem}

\begin{proof} 
\begin{enumerate}[(i)]
\item This follows directly from Lemma \ref{2-33}.

\item Let us note that we can write as 
\begin{align}
T^{(n+1)} = \overline{ T \otimes 1 + 1 \otimes T^{(n)} }
\end{align}
on $ \H \otimes \big( \ot ^n \H \big) $. The self-adjoint operators $ T\otimes I $ and $ I \otimes T^{(n)} $ are strongly commuting. Hence, using Lemma \ref{2-33}, the desired result follows.
\end{enumerate}
\end{proof}

\subsection{Boson creation and annihilation operators}

The boson annihilation operator $ A(f) $ with $ f \in \H $ is defined to be a densely defined closed operator on $ \F _\bb (\H ) $ whose adjoint is given by
\begin{align}
& (A(f)^* \Psi ) ^{(0)} := 0 ,\\
& (A(f) ^* \Psi ) ^{(n)} := \sqrt{n} S_n (f \otimes \Psi ^{(n-1)}) , \q \Psi = \{ \Psi ^{(n)} \} _{n=0} ^\infty  \in D(A(f)^*), \q n\ge 1 .
\end{align}
We note that $ A(f) $ is anti-linear in $ f $ and $ A(g)^* $ linear in $ g $. The boson creation and annihilation operators leave the finite particle subspace
\begin{align}
\F _{\bb , 0} (\H ) := \Big\{ \{ \Psi ^{(n)} \} _{n=0} ^\infty \in \F _\bb (\H ) \, \Big| \,  \Psi ^{(n)} =0 \, \text{for all sufficiently large $n$} \Big\} 
\end{align}
invariant and satisfy the canonical commutation relations:
\begin{align}
[A(f) , A (g ) ^* ] = \left\langle f,g \right\rangle _{\H } , \q [A(f) , A(g ) ] = [A (f) ^* , A(g )^* ]=0, \q f,g \in \H ,
\end{align}
on $ \F _{\bb, 0} (\H ) $.

The following fact is well known.

\begin{Lem}\label{boson-bdd} Let $ K $ be an injective, non-negative, self-adjoint operator on $ \H $. Then, for all $ \Psi \in D(\dd \Gamma _\bb (K) ^{1/2}) $ and $ f\in D(K^{-1/2}) $, 
\begin{align}
\| A(f) \Psi \| _{\F _\bb (\H )} & \le \| K^{-1/2} f \| _\H \| \dd\Gamma _\bb (K) ^{1/2}\Psi \| _{\F _\bb (\H )} , \\
\| A(f)^* \Psi \| _{\F _\bb (\H )} & \le \| K^{-1/2} f \| _\H \| \dd\Gamma _\bb (K) ^{1/2}\Psi \| _{\F _\bb (\H )} + \| f \| _\H \| \Psi \| _{\F _\bb (\H )} .
\end{align}
\end{Lem}
 
\begin{Lem}\label{boson-comm} Let $ T $ be an injective, non-negative, self-adjoint operator on $ \H $. Then, for all $ f \in D(T^{-1/2}) \cap D(T) $, $ A(f) $ and $ A(f) ^* $ map $ D(\dd \Ga _\bb (T) ^{3/2}) $ into $ D(\dd \Ga _\bb (T)) $, and satisfy the following commutation relations:
\begin{align}
[\dd \Ga _\bb (T) , A(f) ^* ] \Psi & = A (T f) ^* \Psi , \\
[\dd \Ga _\bb (T) , A(f)] \Psi & = -A (T f) \Psi ,
\end{align}
for all $ \Psi \in D(\dd \Ga _\bb (T) ^{3/2}) $.
\end{Lem}

\begin{proof} For a proof, see \cite[Theorem 4-27]{AraiFock}.
\end{proof}

\begin{Lem}\label{boson-spec} Let $ T $ be a non-negative self-adjoint operator in $ \H $. Then, the following (i) and (ii) hold.
\begin{enumerate}[(i)]
\item For all $ B_1 , B_2 \in \B ^1 $ and $ f\in R(E_T (B_1)) $, $ A(f) ^* $ maps $ R(E_{\dd \Ga _\bb (T)} (B_2)) \cap D(A(f)^*) $ into $ R(E_{\dd \Ga _\bb (T)} (B_1 + B_2)) $, where $ B_1 + B_2 := \{ \la _1 + \la _2 \in \R \; | \; \la _j \in B_j , \; j=1,2 \} $.

\item For all $ \La \ge 0 $ and $ f\in D(T^{-1/2}) $, $ A(f) $ leave $ R(E_{\dd \Ga _\bb (T)} ([0, \La ])) $ invariant.
\end{enumerate}
\end{Lem}

\begin{proof}
\begin{enumerate}[(i)]
\item Let $ \Psi = \{ \Psi ^{(n)} \} _{n=0} ^\infty \in R(E_{\dd \Ga _\bb (T)} (B_2)) \cap D(A(f)^*) $. By the general theory of direct product operators, it follows that $ E_{\dd \Ga _\bb (T)} (B_2) = \op _{n=0} ^\infty E_{T _\bb ^{(n)}} (B_2) $. Hence, $ \Psi ^{(n)} \in R(E_{T ^{(n) } _\bb} (B_2)) $. By the definition of the creation operator $ A(f) ^* $, 
\begin{align}
(A(f) ^* \Psi ) ^{(n+1)} = \sqrt{n+1} S_{n+1} (f \otimes \Psi ^{(n)}) , \q n \ge 0 .
\end{align}
From Lemma \ref{spec-sum}, we see that the right-hand side belongs to $ R (E_{T _\bb ^{(n)}} (B_1+B_2)) $. Therefore we have $ A(f) ^* \Psi \in R(E_{\dd \Ga _\bb (T)} (B_1+B_2)) $.

\item Let $ f \in D(T^{-1/2}) $ and $ \Psi \in R\big( E_{\dd \Ga _\bb (T)} ([0,\La ])\big) $ for some $ \La \ge 0 $. Since $ D(A(f)) \supset D(\dd \Ga _\bb (T) ^{1/2}) $ from Lemma \ref{boson-bdd}, we see that $ \Psi \in D(A(f)) $. To prove the claim, it is sufficient to see that for all $ \Phi \in R \big( E_{\dd \Ga _\bb (T)} ([\La , \infty )) \big) $, $ \Expect{ \Phi , A(f) \Psi } =0 $.

Now let $ \Phi \in R \big( E_{\dd \Ga _\bb (T)} ([\La , \infty )) \big) $ be fixed arbitrarily and set $ \Phi _n := E_{\dd \Ga _\bb (T)} ([\La , \La +n ]) \Phi \; (n \in \N ) $. Then, $ \Phi _n \to \Phi \; (n \to \infty ) $. Moreover, it follow from (i) that $ A(f) ^* \Phi _n \in R \big( E_{\dd \Ga _\bb (T)} ([\La , \infty )) \big)  $ and thus $ \Expect{A(f) ^* \Phi _n , \Psi } =0 $ for all $ n\in \N $. Hence,
\begin{align}
\Expect{ \Phi , A(f) \Psi } = \lim _{n \to \infty } \Expect{ A(f) ^* \Phi _n,  \Psi } = 0 .
\end{align}
Therefore the assertion follows.
\end{enumerate}
\end{proof}

\subsection{fermion creation and annihilation operators}

The fermion annihilation operator $ B(f) $ with $ f \in \H $ is defined to be a bounded operator on $ \F _\ff (\H ) $ whose adjoint is given by
\begin{align}
& (B(f)^* \Psi ) ^{(0)} := 0 , \\
& (B(f) ^* \Psi ) ^{(n)} := \sqrt{n} A_n (f \otimes \Psi ^{(n-1)}) , \q \Psi = \{ \Psi ^{(n)} \} _{n=0} ^\infty \in \F _\ff (\H ) ,
\end{align}
where $ A_n $ denotes the anti-symmetrization operator on $ \otimes ^n \H $, i.e. $ A_n (\otimes ^n \H ) = \wedge ^n \H $. It is well known that, the operator norm of $ B(f) ^\sharp $ becomes
\begin{align}\label{fermion-bdd}
\| B(f) ^\sharp \| = \| f \| _\H .
\end{align}
$ B(f) $ is anti-linear in $ f $ and $ B(f)^* $ linear in $ f $. The fermion creation and annihilation operators satisfy the canonical anti-commutation relations:
\begin{align}
\{ B(f) , B (g ) ^* \} = \left\langle f,g \right\rangle _{\H } , \q \{B(f) , B(g ) \} = \{B (f) ^* , B(g )^* \} =0, \q f,g \in \H ,
\end{align}
on $ \F _{\ff} (\H ) $, where $ \{ X,Y \}:= XY+YX $.

We define an operator-valued function $ \psi (\cdot \, , \cdot ) $ by
\begin{align}
\psi (f,g) := B(f) + B(g) ^* , \q f,g \in \H .
\end{align}
Let $ \E _\ff $ be the set consisting of finite linear combinations of finite products of operators $ \psi (f,g) \; (f,g\in \H ) $. For a product operator $ \psi (f_1,g_1) \cdots \psi (f_n,g_n) \; (f_j , g_j \in \H , \; j=1,...,n , \; n\ge 1) $, we define the \textit{normal ordering} $ \Normal{ \psi (f_1,g_1) \cdots \psi (f_n,g_n) } $ by
\begin{align}
\Normal{ \psi (f_1,g_1) \cdots \psi (f_n,g_n) } \, = \sideset{}{'} \sum _k \sgn (\sigma ) B(g_{i_1}) ^* \cdots B(g_{i_k}) ^* B(f_{j_1}) \cdots B(f_{j_{n-k}}) ,
\end{align}
where the symbol $ \sum _k ' $ denotes the sum over $ i_1 ,...,i_k , j_1,...,j_{n-k} $ satisfying $ i_1 <\cdots < i_k $, $ j_1 <\cdots < j_{n-k} $, $ \{ i_1 , ... , i_k \} \cap \{ j_1 , ... , j_{n-k} \} = \emptyset $, $ \{ i_1 , ... , i_k \} \cup \{ j_1 , ... , j_{n-k} \} = \{ 1,...,n \} $, and $ \sigma $ is the permutation $ (1,...,n) \mapsto (i_1,...,i_k , j_1,..., j_{n-k}) $. We extend it by linearity to $ \E _\ff $.

\begin{Lem}\label{fermion-comm} Let $ T $ be a self-adjoint operator in $ H $. Then, for all $ f \in D(T) $, $ B(f) $ and $ B(f) ^* $ leave $ D(\dd \Ga _\ff (T) ) $ invariant, and satisfy the following commutation relations:
\begin{align}
[\dd \Ga _\ff (T) , B(f) ^* ] \Psi & = B (T f) ^* \Psi , \label{f-comm1} \\
[\dd \Ga _\ff (T) , B(f)] \Psi & = - B (T f) \Psi , \label{f-comm2}
\end{align}
for all $ \Psi \in D(\dd \Ga _\ff (T) ) $.
\end{Lem}

\begin{proof} See \cite[Theorem 5-9]{AraiFock}.
\end{proof}

\begin{Lem}\label{fermion-spec} Let $ T $ be a non-negative self-adjoint operator in $ \H $. Then, the following (i) and (ii) hold.
\begin{enumerate}[(i)]
\item For all $ a,R \ge 0 $ and $ f\in R(E_T ([0, a])) $, $ B^* (f) $ maps $ R(E_{\dd \Ga _\ff (T)} ([0, R])) $ into $ R(E_{\dd \Ga _\ff (T)} ([0, R+a])) $.

\item For all $ R \ge 0 $ and $ f\in \H $, $ B(f) $ leave $ R(E_{\dd \Ga _\ff (T)} ([0 , R])) $ invariant.
\end{enumerate}
\end{Lem}

\begin{proof} Similar to the proof of Lemma \ref{boson-spec}.
\end{proof}

\subsection{Electromagnetic fields}

Next, we introduce the photon field quantized in the Coulomb gauge. We adopt as the one-photon Hilbert space
\begin{align}
\H _\mathrm{ph} := L^2 (\R ^3 _\kk ; \C ^2) .
\end{align}
The above $ \R ^3 _\kk := \{ \kk = (k^1,k^2,k^3) \, | \, k^j \in \R , \; j=1,2,3 \} $ physically represents the momentum space of photons. If there is no confusion, we omit the subscript $ \kk $ in $ \R ^3 _\kk $. We freely use the identification $ L^2 (\R ^3 _\kk ; \C ^2) = \op ^2 L^2 (\R ^3 _\kk ) $. The Hilbert space for the quantized electromagnetic field is given by $ \F _\bb (\H _\rm{ph} ) $ the boson Fock space over $ \H _\rm{ph} $.

The energy of a photon with momentum $ \kk \in \R ^3 $ is given by $ \omega (\kk ) := |\kk | $. Then the function $ \omega $ defines uniquely a multiplication operator on $ \H _\rm{ph} $ which is injective, non-negative and self-adjoint. We denote it by the same symbol $ \omega $ also. The free Hamiltonian of the quantum electromagnetic field is given by the second quantization of $ \omega $:
\begin{align}
H _\mathrm{ph} := \mathrm{d}\Gamma _\bb ( \omega ) : \F _\bb (\H _\rm{ph} ) \to \F _\bb (\H _\rm{ph} ) .
\end{align}

We denote by $ a (\cdot ) \; ( \cdot \in \H _\mathrm{ph}) $ the annihilation operator on $ \F _\rm{ph} $. For each $ f \in L^2 (\R _\kk ^3) $, we use the notation:
\begin{align}
& a^{(1)} (f) := a (f,0) , \q a^{(2)} (f) := a (0,f) .
\end{align}

For $ \chi _\rm{ph} \in L^2 (\R ^3 _\xx ) $ satisfying $ \chi _\rm{ph} ^* = \chi _\rm{ph} $ and $ \hat{\chi _\rm{ph}} / \sqrt{\omega} \in L^2 (\R ^3 _\kk ) $, we set
\begin{align}\label{EM-def}
& A_j (0, \xx ) := \sum _{r=1,2} \Big( a^{(r)} \Big( \frac{\hat{\chi _\rm{ph} ^\xx } e^{(r)} _j }{\sqrt{2\omega }} \Big) + a^{(r)} \Big( \frac{\hat{\chi _\rm{ph} ^\xx} e^{(r)}_j }{\sqrt{2\omega }} \Big) ^* \Big) , \\
& \chi _\rm{ph} ^\xx (\yy ) := \chi _\rm{ph} (\yy - \xx ) , \q \yy \in \R ^3 ,
\end{align}
where $ \hat{\chiph } $ denotes the Fourier transform of $ \chiph $, and $ \chiph ^* $ denotes the complex conjugate. The functions $ \mathbf{e} ^{(r )} (\kk ) = (e _j ^{(r )} (\kk )) _{j=1} ^3 \in \R ^3 , \, r =1,2, $ are the polarization vectors satisfying
\begin{align}
\ee ^{(r)} (\kk ) \cdot \ee ^{(r')} (\kk ) = \delta _{rr'} , \q \kk \cdot \ee ^{(r)} (\kk ) =0 , \q \rm{a.e.} \; \kk \in \R ^3, \q r,r' = 1,2.
\end{align}
$ A_j (0, \xx ) $ is called the point-like quantized electromagnetic field at time $ t=0 $ with momentum cutoff $ \hat{\chi _\mathrm{ph}} $. As is well-known, $ A_j (0, \xx ) \; ( j =1,2,3) $ are essentially self-adjoint. We denote the closure of $ A_j (0, \xx ) $ by the same symbol. 

We assume the following condition.

\begin{Hypo}\label{B-Hypo} $ \hat{\chiph} / \omega \in L^2 (\R ^3 _\kk ) $.
\end{Hypo}

\begin{Lem}\label{gauge-bdd} Under Hypothesis \ref{B-Hypo}, for all $ i=1,2,3 $, $ \xx \in \R ^3 $ and $ \Psi \in D(H_\rm{ph} ^{1/2}) $,
\begin{align}
\| A_i (0, \xx ) \Psi \| \le M_\rm{ph} \| (H_\rm{ph} + 1)^{1/2} \Psi \| ,
\end{align}
where $ M_\rm{ph} := 2\sqrt{2} \| \hat{\chiph } /\omega \| _{L^2 (\R ^3 _\kk )} + \sqrt{2} \| \hat{\chiph }/\sqrt{\omega} \big\| _{L^2 (\R ^3 _\kk )} $.
\end{Lem}

\begin{proof} This is a simple application of Lemma \ref{boson-bdd}.
\end{proof}

\begin{Rem}\normalfont If the momentum cutoff function $ \hat{\chiph } $ is taken to be the characteristic function of the set $ \{ \kk \in \R ^3 \, \big| \, |\kk | \le \La _0 \} $, then this satisfies Hypothesis \ref{B-Hypo}.
\end{Rem}

\subsection{Dirac fields}

We define the quantized Dirac field. We adopt as the one-electron Hilbert space
\begin{align}
\H _\mathrm{el} := L^2 (\R ^3 _\pp ; \C ^4) ,
\end{align}
where $ \R ^3 _\pp := \{ \mathbf{p} = (p^1,p^2,p^3) \, | \, p^j \in \R , \, j=1,2,3 \} $ physically represents the momentum space of electrons. The Hilbert space for the quantized Dirac field is given by $ \F _\ff (\H _\rm{el}) $ the fermion Fock space over $ \H _\mathrm{el} $.

We denote the mass of the Dirac particle by $ M>0 $. One-electron Hamiltonian in $ \H _\rm{el} $ is the multiplication operator by the function $ \textstyle E_M (\pp ) := \sqrt{\pp ^2 + M^2} \, \, (\pp \in \R ^3) $. The Hamiltonian of the free quantum Dirac field is given by
\begin{align}
H_\mathrm{el} := \mathrm{d}\Gamma _\ff (E_M ) : \F _\ff (\H _\rm{el}) \to \F _\ff (\H _\rm{el}), 
\end{align}
the fermion second quantization operator of $ E_M : \H _\mathrm{el} \to \H _\mathrm{el} $. The operator $ H_\rm{el} $ is non-negative and self-adjoint.

Let $ \gamma ^\mu \, (\mu =0,1,2,3) $ be $ 4\times 4 $ gamma matrices, i.e., $ \gamma ^0 $ is hermitian and $ \gamma ^j \, (j=1,2,3) $ are anti-hermitian, satisfying
\begin{align}
\{ \gamma ^\mu , \gamma ^\nu \} =2g^{\mu\nu}, \q \mu, \nu =0,1,2,3.
\end{align}
Let $ \alpha ^\mu := \gamma ^0 \gamma ^\mu , \, \beta := \gamma ^0 $, and let $ s_1 := \frac{i}{2} \gamma ^2 \gamma ^3 , \, s_2 := \frac{i}{2} \gamma ^3 \gamma ^1 , \, s_3 := \frac{i}{2} \gamma ^1 \gamma ^2 $. Let $ u_s (\pp ) = (u_s ^l (\pp )) _{l=1} ^4 \in \C ^4 $ describe the positive energy part with spin $ s = \pm 1/2 $ and $ v_s (\pp ) = (v_s ^l  (\pp )) _{l=1} ^4 \in \C ^4 $ the negative energy part with spin $ s $, that is,
\begin{align}
& ( { \boldsymbol \alpha } \cdot \pp + \beta M) u_s (\pp ) = E_M (\pp ) u_s (\pp ) , \q ( \mathbf{s} \cdot \pp ) u_s (\pp ) = s|\pp | u_s (\pp ) , \\
& ( { \boldsymbol \alpha } \cdot \pp + \beta M) v_s (\pp ) = -E_M (\pp ) v_s (\pp ) , \q ( \mathbf{s} \cdot \pp ) v_s (\pp ) = s|\pp | v_s (\pp ) , \q \pp \in \R ^3 .
\end{align}
These form an orthogonal base of $ \C ^4 $,
\begin{align}
u_s (\pp ) ^* u_{s'} (\pp ) = v_s (\pp ) ^* v_{s'} (\pp ) = \delta _{ss'} , \q u_s (\pp ) ^* v_{s'} (\pp ) = 0 ,\q \pp \in \R ^3,
\end{align}
and satisfy the completeness,
\begin{align*}
\sum _s \big( u^l_s (\pp ) u^{l' } _s (\pp ) ^* + v^l _s(\pp ) v ^{l' } _s(\pp ) ^*  \big) = \delta _{ll' },\q \pp \in \R ^3 .
\end{align*}

We denote by $ B(\cdot ) \; ( \cdot \in \H _\rm{el}) $ the annihilation operator on $ \F _\ff (\H _\rm{el}) $. For each $ g \in L^2 (\R _\pp ^3) $, we use the notation 
\begin{align*}
& b _{1/2} (g) := B(g,0,0,0) , && b_{-1/2} (g) := B(0,g,0,0) , \\
& d _{1/2} (g) := B(0,0,g,0) , && d_{-1/2} (g) := B(0,0,0,g) ,
\end{align*}
Then, we have the canonical anti-commutation relations: 
\begin{align}
& \{ b_s (g) , b_{s'} (g')^* \} =\{ d_s (g) , d_{s'} (g')^* \} = \delta _{ss'} \left\langle g , g' \right\rangle _{L^2 (\R ^3 _\pp )} , \no \\
& \{ b_s (g) , b_{s'} (g') \} =\{ d_s (g) , d_{s'} (g') \} = \{ b_s (g) , d_{s'} (g') \} =\{ b_s (g) , d_{s'} (g')^* \} =0.
\end{align}

Fix $ \chi _\rm{el} \in L^2 (\R ^3 _\xx ) $ satisfying $ \chi _\rm{el} ^* = \chi _\rm{el} $, and set 
\begin{align}
& \psi _l (0, \xx ) := \sum _{s= \pm 1/2} \Big( b_s \Big( \hat{\chi _\rm{el} ^\xx } (u^l _s) ^* \Big) + d_s \Big( \hat{\chi _\rm{el} ^\xx } \, \, \widetilde{v} ^l_s \Big) ^* \Big) , \\
& \chi _\rm{el} ^\xx (\yy ) := \chi _\rm{el} (\yy - \xx ) , \q \yy \in \R ^3 ,
\end{align}
where $ \widetilde{v} _s ^l (\pp ) := v_s ^l (-\pp ) $. $ \psi _l (0, \xx ) $ is called the point-like quantized Dirac field at time $ t =0 $ with momentum cutoff $ \hat{\chi _\mathrm{el}} $. For each $ \xx \in \R ^3 $ and $ \mu =0,1,2,3 $, we define the current operator $ j^\mu (0, \xx ) $ by
\begin{align}
j^\mu (0, \xx ) := \sum _{l,l' =1} ^4  \psi  _l (0, \xx )^* \alpha ^\mu _{ll'} \psi _{l'} (0, \xx ) .   
\end{align}
Then $ j^\mu (0, \xx ) $ is bounded and self-adjoint.

\begin{Lem}\label{Dirac-bdd} For all $ \mu =0,1,2,3 $ and $ \xx \in \R ^3 $,
\begin{align}
\| j^\mu (0, \xx ) \| \le M_\rm{cu} ,
\end{align}
where $ M_\rm{cu} := 256 \| \hat{\chiel} \| ^2  _{L^2 (\R ^3 _\pp )} $.
\end{Lem}

\begin{proof} A simple application of \eqref{fermion-bdd}.
\end{proof}

\subsection{Total Hamiltonian}\label{subsection7}

The state space for QED in Coulomb gauge is taken to be 
\begin{align}
\F _\rm{tot} := \F _\ff (\H _\rm{el}) \ot \F _\bb ( \H _\rm{ph} ).
\end{align}
The free Hamiltonian is 
\begin{align}
H _\rm{fr} := H_\rm{el} \ot I + I \ot H _\rm{ph} ,
\end{align}
where the subscript $ \rm{fr} $ in $ H_\rm{fr} $ means \textit{free}.

We denote the charge of the Dirac particle by $ e\in \R $. Let $ \chisp \in L^1 (\R ^3) $ be a real valued function on $ \R ^3 $ playing the role of spacial cut-off. The first interaction term $ H _{\rm{I}} $ is defined as 
\begin{align}
& D(H _{\rm{I}}) = D((I \otimes H_\rm{ph} )^{1/2}) , \no \\
& H _{\rm{I} } \Psi = e \sum _{i=1} ^3 \int _{\R ^3} d\xx \, \chi _{\mathrm{sp}} (\xx ) \Normal{ j^i (0, \xx ) } \ot A _i (0, \xx ) \Psi , \q \Psi \in D(H _{\rm{I}}) ,
\end{align}
where the integral on the right hand side is a strong Bochner integral. We adopt the Coulomb term $ H _{\rm{II}} $ which is given by
\begin{align}
& D(H_\rm{II}) := \F _\rm{tot} , \no \\
& H_{\rm{II}} := \frac{e^2}{2} \int _{\R ^3 \times \R ^3} d\xx d\yy \, \chisp (\xx ) \chisp (\yy ) V_C (\xx-\yy) \Normal{ j^0 (0, \xx ) j ^0 (0, \yy ) } \otimes I , \label{DefII}
\end{align}
with
\begin{align}
V_C (\xx - \yy ):= \frac{1}{4\pi}\int _{\R ^3} \frac{d\kk }{\omega (\kk ) ^2} |\hat{\chiph } (\kk )| ^2 e^{i\kk (\xx - \yy )},
\end{align}
where the integral on the right-hand side of \eqref{DefII} is a Bochner integral with respect to the operator norm. The well-definedness of $ H_\rm{I} $ and $ H_\rm{II} $ is proven in later (see Lemma \ref{s.a.Lem}). Then, $ H_\rm{I} $ is symmetric, and $ H_\rm{II} $ is bounded and self-adjoint. We remark that the interaction potential $V_C (\xx-\yy) $ converges to the familiar Coulomb potential 
\[ \frac{1}{4\pi}\frac{1}{|\xx -\yy|} \]
in the distribution sense as the photon UV cutoff $ \hat{\chiph} $ is removed.
Finally, the interaction Hamiltonian $ H_\rm{int} $ and the total Hamiltonian $ H_\rm{tot} $ is defined by
\begin{align}
& H_\rm{int} := \overline{H _{\rm{I}} } + H _{\rm{II} } ,\\
& H_\rm{tot} := H _\mathrm{fr} + H_\rm{int} .
\end{align}

\subsection{Self-adjointness}\label{subsection8}

\begin{Lem}\label{s.a.Lem} Assume Hypothesis \ref{B-Hypo}. Then, the following (i)-(iii) hold:
\begin{enumerate}[(i)]
\item For all $ \Psi \in D((I \otimes H_\rm{ph}) ^{1/2}) $,
\begin{align*}
\sum _{i=1} ^3 \int _{\R ^3} d\xx \, |\chisp (\xx )| \, \| \Normal{ j^i (0,\xx ) } \otimes A_i (0,\xx ) \Psi \| \le 3 \| \chisp \| _{L^1 (\R ^3)} M_\rm{cu} M_\rm{ph} \| (I \otimes H_\rm{ph} +1 ) ^{1/2} \Psi \| < \infty .
\end{align*}

\item It follows that
\begin{align*}
\int _{\R ^3 \times \R ^3} d\xx d\yy \, | \chisp (\xx ) \chisp (\yy ) V_C (\xx-\yy) | \, \| \Normal{ j^0 (\xx ) j ^0 (\yy ) } \otimes I \| \le \| \chisp \| _{L^1 (\R ^3)} ^2 M_C M_\rm{cu} ^2 < \infty , 
\end{align*}
where $ M_C := (1/4\pi ) \| \hat{\chiph} / \omega \| _{L^2 (\R ^3 _\kk )} $.

\item $ H_\rm{int} $ is $ H_\rm{fr} ^{1/2} $-bounded, closed and symmetric.

\item $ H_\rm{tot} $ is self-adjoint on $ D(H_\rm{fr}) $, and bounded from below.
\end{enumerate}
\end{Lem}

\begin{proof} (i) and (ii) follow from Lemma \ref{gauge-bdd} and \ref{Dirac-bdd}.

We prove (iii). It is easy to see that $ H_\rm{I} $ and $ H_\rm{II} $ are symmetric. By (i), $ H_\rm{I} $ is $ H_\rm{fr} ^{1/2} $-bounded. By (ii), $ H_\rm{II} $ is bounded. Thus, $ H_\rm{int} $ is $ H_\rm{fr} ^{1/2} $-bounded, closed and symmetric.

By (iii), $ H_\rm{int} $ is infinitesimally $ H_\rm{fr} $-bounded. Thus, (iv) follows from the Kato-Rellich theorem.
\end{proof}

\subsection{Time-ordered exponential on the complex plane}\label{QED-timeorder}

Basic hypothesis to apply our abstract theory is:

\begin{Hypo}[Ultraviolet cutoff]\label{UV-Hypo} There exists constants $ \La _\rm{el} , \La _\rm{ph} \ge 0 $ such that $ \supp \hat{\chiel} \subset \{ |\pp | \le \La _\rm{el} \} , \; \supp \hat{\chiph} \subset \{ |\kk | \le \La _\rm{ph} \} $.
\end{Hypo}

\begin{Lem}\label{tensor-spec} Let $ K_j \; (j=1,...,n, \; n \ge 1) $ be non-negative self-adjoint operators, and $ B_j $ be closable operators on Hilbert spaces $ \H _j $. Suppose that for each $ j $, there exists a constant $ a_j \ge 0 $ such that, for all $ L \ge 0 $, $ B_j $ maps $ R(E_{K_j} ([0,L ])) $ into $ R(E_{K_j} ([0,L+ a_j ])) $. Then, for a self-adjoint operator
\begin{align}
K := \sum _{j=1} ^n I \otimes \cdots \otimes I \otimes \stackrel{j\text{-th}}{K_j} \otimes I \otimes \cdots \otimes I
\end{align}
on $ \ot _{j=1} ^n \H _j $, the tensor product operator $ B_1 \otimes \cdots \otimes B_n $ maps $ R(E_{K} ([0,L ])) $ into $ R(E_{K} ([0,L+ \sum _j a_j ])) $.
\end{Lem}

\begin{proof} For each $ L \ge 0 $, set 
\begin{align}
J_L := \big\{ (\la _1 , ... , \la _n ) \in [0, \infty ) ^n \; \big| \; \sum _j \la _j \in [0, L ] \big\} \subset \R ^n .
\end{align}
Then, for all $ \e >0 $ and $ L \ge 0 $, there exist $ n $ dimensional half-closed intervals $ I^{(k)} _\e := I^{(k)} _{\e , 1} \times \cdots \times I^{(k)} _{\e , n} \subset \R ^n \; \big( I^{(k)} _{\e , j} = [L _{\e , j} ^{(k)} , \widetilde{L } _{\e , j} ^{(k)} \big) \subset \R , \; j=1,..., n, \; k=1,..., N_\e , \; 1\le N_\e <\infty ) $ such that $ I^{(k)} _{\e } \cap I^{(k')} _{\e } = \emptyset \; (k \neq k') $ and 
\begin{align}\label{JIJ2}
J_L \subset \bigcup _{k=1} ^{N_\e } I^{(k)} _{\e } \subset J_{L + \e } .
\end{align}
Now, we set $ \widetilde{K} _j := I \otimes \cdots \otimes I \otimes K_j \otimes I \otimes \cdots \otimes I $. Then, $ \widetilde{K} _j \; (j=1,...,n) $ are strongly commuting self-adjoint operators; it follows from Lemma \ref{2-33} that
\begin{align}\label{JIJ3}
R \big( E_{K} ([0,L]) \big) = R \big( (E_{\widetilde{K} _1 } \times \cdots \times E_{\widetilde{K} _j } ) (J_L) \big) .
\end{align}
Using \eqref{JIJ2}, we have
\begin{align}\label{JIJ4}
R \big( (E_{\widetilde{K} _1 } \times \cdots \times E_{\widetilde{K} _j } ) (J_L) \big) \subset \op _{k=1} ^{N_\e } \Big( R \big( E_{\widetilde{K}_1} (I_{\e , 1} ^{(k)}) \big) \otimes \cdots \otimes  R \big( E_{\widetilde{K}_n} (I_{\e , n} ^{(k)}) \big) \Big) .
\end{align}
By the present assumption, we see that $ B_1 \otimes \cdots \otimes B_n $ maps $ R \big( E_{\widetilde{K}_1} (I_{\e , 1} ^{(k)}) \big) \otimes \cdots \otimes  R \big( E_{\widetilde{K}_n} (I_{\e , n} ^{(k)}) \big)  $ into $ R \big( E_{\widetilde{K}_1} ([0, \widetilde{L} _{\e , 1} ^{(k)} + a_1]) \big) \otimes \cdots \otimes  R \big( E_{\widetilde{K}_n} ([0, \widetilde{L} _{\e , n} ^{(k)} + a_n]) \big)  $ for each $ k=1,..., N_\e $. Combining this with \eqref{JIJ2}-\eqref{JIJ4}, we conclude that for all $ \e >0 $, $ B_1 \otimes \cdots \otimes B_n $ maps $ R \big( E_{K} ([0,L]) \big) $ into $ R \big( E_{K} ([0,L+ \e + \sum _j a_j ]) \big) $, that is, for all $ \Psi \in R \big( E_{K} ([0,L]) \big) $,
\begin{align}\label{EAA}
E_{K} \big( [0,L+ \e + \sum _j a_j ] \big) (B_1 \otimes \cdots \otimes B_n) \Psi = (B_1 \otimes \cdots \otimes B_n) \Psi , \q \e >0 .
\end{align}
We can take the limit $ \e \downarrow 0 $ in \eqref{EAA} since the projection-valued function $ E_{K} ([0, L ]) \; (L \ge 0) $ is right-continuous with respect to $ L $. Thus, we obtain $ E_{K} ([0,L+ \sum _j a_j ]) (B_1 \otimes \cdots \otimes B_n) \Psi = (B_1 \otimes \cdots \otimes B_n) \Psi $, and the desired result follows.
\end{proof}

In what follows, we use the following notations:
\begin{align}
& \F _E := R (E_{H_\rm{fr}} ([0, E])) , \q E\ge 0 , \\
& \F _\rm{fin} := \bigcup _{E\ge 0} \F _E . \label{Ffin} 
\end{align}

In order to construct the time-ordered exponential, it is sufficient to see that Theorem \ref{time-ordered-exp} can be applied to our case by checking that $ H_\rm{int} $ is in $ \mathcal{C}_0 $-class with respect to $ H_\rm{fr} $. The correspondence of the symbols is as follows: $ H_0 = H_\rm{fr}, \; H_1 = H_\rm{int} , \; V_E = \F _E , \; D_\rm{fin} = \F _\rm{fin} $.

\begin{Lem}\label{field-spec} Assume Hypothesis \ref{UV-Hypo}. Then, the following (i) and (ii) hold.
\begin{enumerate}[(i)]
\item For all $ E \ge 0 $, $ \xx \in \R ^3 $ and $ j=1,2,3 $, $ A_j (0,\xx ) $ maps $ R \big( E_{H_\rm{ph}} ([0,E]) \big) $ into $ R \big( E_{H_\rm{ph}} ([0,E + \La _\rm{ph}]) \big) $.

\item For all $ E \ge 0 $, $ \xx \in \R ^3 $ and $ l=1,2,3,4 $, $ \psi _l (0, \xx ) $ and $ \psi _l (0, \xx ) ^* $ map $ R \big( E_{H_\rm{el}} ([0,E]) \big) $ into $ R \big( E_{H_\rm{el}} ([0,E + \sqrt{\La _\rm{el} + M^2} ]) \big) $.

\item $ I \otimes A_j (0, \xx ) $ is in $ \mathcal{C} _0 $-class with $ H_0 = H_\rm{fr} $.

\item $ \psi _l (0, \xx ) \otimes I $ is in $ \mathcal{C} _0 $-class with $ H_0 = H_\rm{fr} $.
\end{enumerate}
\end{Lem}

\begin{proof} 
\begin{enumerate}[(i)]
\item Let us recall the definition of the quantized electromagnetic field \eqref{EM-def}. Under the Hypothesis \ref{UV-Hypo}, it is easy to see that $ \big( \hat{\chiph ^\xx } e^{(1)} _j / \sqrt{\omega } , \, \hat{\chiph ^\xx } e^{(r)} _j / \sqrt{\omega } \big) \in R \big( E_\omega ([0, \La _\rm{ph}])\big) $. Hence, using Lemma \ref{boson-spec}, the assertion follows.

\item Similar to the proof of (i).

\item By Lemma \eqref{gauge-bdd}, $ I \otimes A_j (0, \xx ) $ is $ H_\rm{fr} ^{1/2} $-bounded. Combining (i) and \ref{tensor-spec}, we see that for all $ E \ge 0 $, $ A_j (0,\xx ) $ maps $ R \big( E_{H_\rm{fr}} ([0,E]) \big) $ into $ R \big( E_{H_\rm{fr}} ([0,E + \La _\rm{ph}]) \big) $. Therefore, the assertion follows.

\item Similar to the proof of (iii).
\end{enumerate}
\end{proof}

\begin{Lem}\label{QEDass1} Assume Hypotheses \ref{B-Hypo} and \ref{UV-Hypo}. Then, the following (i) and (ii) hold:
\begin{enumerate}[(i)]
\item For all $ E \ge 0 $, $ H_\rm{I} $ maps $ \F _E $ into $ \F _{E+2 \sqrt{\La _\rm{el} ^2 + M^2} + \La _\rm{ph} } $. 

\item For all $ E \ge 0 $, $ H_\rm{II} $ maps $ \F _E $ into $ \F _{E+4 \sqrt{\La _\rm{el} ^2 + M^2} } $. 

\item $ H_\rm{int} $ is in $ \mathcal{C} _0 $-class with $ H_0 = H_\rm{fr} $.
\end{enumerate}
\end{Lem}

\begin{proof}
\begin{enumerate}[(i)] 
\item One can see that, for all $ E \ge 0 $, $ \xx \in \R ^3 $ and $ \mu =0,1,2,3 $, $ \Normal{ j^\mu (0, \xx ) } $ maps $ R \big( E_{H_\rm{el}} ([0,E]) \big) $ into $ R \big( E_{H_\rm{el}} ([0,E + 2 \sqrt{\La _\rm{el} + M^2} ]) \big) $ in the same manner as Lemma \ref{field-spec} (ii). Now, fix $ E \ge 0 $ arbitrarily, and let $ \Psi \in \F _E $. Applying Lemmas \ref{tensor-spec} and \ref{field-spec}, we see that $ \Normal{ j^i (0,\xx )} \otimes A_i (0, \xx ) \Psi \in \F _{E+2 \sqrt{\La _\rm{el} ^2 + M^2} + \La _\rm{ph} } $ for all $ \xx \in \R ^3 $ and $ i=1,2,3 $. Hence, we have $ H_\rm{I} \Psi \in \F _{E+2 \sqrt{\La _\rm{el} ^2 + M^2} + \La _\rm{ph} } $ because $ \F _{E+2 \sqrt{\La _\rm{el} ^2 + M^2} + \La _\rm{ph} } $ is closed subspace. Thus, the assertion follows.

\item Similar to the proof of (i).

\item This follows from (i), (ii) and Lemma \ref{s.a.Lem}.
\end{enumerate}
\end{proof}

From Lemma \ref{QEDass1}, we can apply the abstract theory constructed in the previous sections to obtain:

\begin{Thm}\label{QEDexp} Assume Hypotheses \ref{B-Hypo} and \ref{UV-Hypo}. Take a piecewisely continuously differentiable simple curve $\Gamma_{z,z'}$ which starts at $z'$ and ends at $z$ with $\im z'\le\im z$. Then,
\begin{align}
\F _\fin \subset D\left(T\exp\left(-i\int_{\Gamma_{z,z'}} d\zeta H_\rm{int} (\zeta)\right)\right) , 
\end{align}
where
\begin{align}
H_\rm{int} (z) := e^{izH_\rm{fr}} H_\rm{int} e^{-izH_\rm{fr}} , \q z\in \C .
\end{align}
Furthermore, $ T\exp \big( -i\int_{\Gamma_{z,z'}} d\zeta H_\rm{int} (\zeta) \big) $ has properties stated in Theorems \ref{time-ordered-exp}-\ref{Time-order-ext}, with $ H_0 $ replaced by $ H_\rm{fr} $, $ H_1 $ by $ H_\rm{int} $ and $ D_\rm{fin} $ by $ \F _\rm{fin} $.
\end{Thm}


\subsection{Gell-Mann -- Low formula for QED}

To apply our abstract theory, we need some preliminaries. For two linear operators $ A $ and $ B $ in a Hilbert space $ \H $, we define $ \rm{ad} _A ^k (B) , \; (k= 0,1,2,...)  $ by
\begin{align}
& \rm{ad} _A ^0 (B) := B , \\
& \rm{ad} _A ^k (B) := [A, \rm{ad} _A ^{k-1} (B) ] , \q k \ge 1.
\end{align}
It is easy to see that, for all integer $ n \ge 0 $,
\begin{align}
A^n B \psi = \sum _{k=0} ^n {}_n C _k \, \rm{ad} _A ^k (B) A^{n-k} \Psi , \q \Psi \in \bigcap _{k=0} ^n D(A^k B A^{n-k}) .
\end{align}

\begin{Lem}\label{prelimLem1} Let $ n_0 \ge 0 $ be an integer and $ r \ge 0$ a real number. Let $ T_0 $ be a self-adjoint operator and $ T_1 $ a densely defined closed operator on a Hilbert space $ \H $. Suppose that there exists a subspace $ D \subset \H $ having the following properties (I)-(III):
\begin{enumerate}[(I)]
\item $ T_0 $ and $ T_1 $ leave $ D $ invariant.

\item $ D $ is a core of $ T_0 ^{n_0 +r } $.

\item For all $ n = 0,..., n_0 $, $ \rm{ad} _{T_0} ^{n} (T_1) $ is $ T_0 ^{n+r} $-bounded on $ D $, i.e., there exist constants $ C_1 , C_2 \ge 0 $ such that for all $ \Psi \in D $,
\begin{align}
\| \rm{ad} _{T_0} ^{n} (T_1) \Psi \| \le C_1 \| T_0^{n+r} \Psi \| + C_2 \| \Psi \| . 
\end{align}
\end{enumerate}
Then, for all $ n=0,...,n_0 $, $ T_1 $ maps $ D(T_0^{n+r}) $ into $ D(T_0^n) $.
\end{Lem}

\begin{proof} We prove (i). Let $ n=0,...,n_0 $ be fixed arbitrarily. By the condition (I), for all $ \Psi \in D $,
\begin{align}
T_0 ^n T_1 \Psi = \sum _{k=0} ^n {}_n C_k \, \rm{ad} _{T_0} ^k (T_1) T_0 ^{n-k} \Psi .
\end{align}
By the condition (III), each of $ \rm{ad} _{T_0} ^k (T_1) T_0^{n-k}  \; (k=0,...,n) $ is $ T_0^{n+r} $-bounded on $ D $, and thus, so is $ T_0 ^n T_1 $. Hence, there exist constants $ C_1 , C_2 \ge 0 $ such that
\begin{align}
\| T_0 ^n T_1 \Psi \| \le C_1 \| T_0 ^{n+r} \Psi \| + C_2 \| \Psi \| , \q \Psi \in D .
\end{align}
Let us note that $ T_1 $ is $ T_0^r $-bounded on $ D $ from the condition (III). Using the condition (II) and the closedness of $ T_0^n $, we see that the above $ \Psi $ can be extended onto $ D(T_0^{n+r}) $. Thus, the assertion follows.
\end{proof}

\begin{Lem}\label{prelimLem1'} Let $ n_0 \ge 0 $ be an integer. Let $ T_0 $ be a self-adjoint operator and $ T_1 $ a closed symmetric operator on a Hilbert space $ \H $. Suppose that there exists a subspace $ D \subset \H $ having the following properties (I)-(III).
\begin{enumerate}[(I)]
\item $ T_0 $ and $ T_1 $ leave $ D $ invariant.

\item $ D $ is a core of $ T_0 ^{n_0 +1} $.

\item For all $ n = 0,..., n_0 $, $ \rm{ad} _{T_0} ^{n} (T_1) $ is infinitesimally $ T_0 ^{n+1} $-bounded on $ D $, i.e., for all $ \e >0 $, there exists a constant $ C_\e \ge 0 $ such that for all $ \Psi \in D $,
\begin{align}
\| \rm{ad} _{T_0} ^{n} (T_1) \Psi \| \le \e \| T_0^{n+1} \Psi \| + C_\e \| \Psi \| . 
\end{align}
\end{enumerate}
Then, $ T:=T_0+T_1 $ is self-adjoint. Furthermore, for all $ n =1,..., n_0+1 $, $ T^n - T_0 ^n $ is infinitesimally $ T_0 ^n $-bounded, and it follows that 
\begin{align}\label{domain1}
D(T^n) = D(T_0 ^n) .
\end{align}
\end{Lem}

\begin{proof} From the conditions (II) and (III) for $ n=1 $, $ T_1 $ is infinitesimally $ T_0 $-bounded. Hence, it follows from the Kato-Rellich theorem that $ T $ is self-adjoint and 
\begin{align}
D(T) = D(T_0).
\end{align}

By Lemma \ref{prelimLem1}, for all $ n= 1,..., n_0 $, $ T_1 $ maps $ D(T_0 ^{n+1}) $ into $ D(T_0 ^n) $. Hence, we have 
\begin{align}\label{domain10}
D(T^{n}) \supset D(T_0 ^{n}) , \q n=1, ..., n_0 +1 .
\end{align}

We prove the remaining claim by induction. The case $ n=1 $ has already been proved. Suppose that the claim is true for some $ n < n_0+1 $. By the condition (I), we have
\begin{align}\label{domain2}
T^{n+1} \Psi - T_0 ^{n+1} \Psi = T_0 ^n T_1 \Psi + (T^n - T_0 ^n) T \Psi , \q \Psi \in \D . 
\end{align}
From the induction hypothesis, for all $ \e >0 $,
\begin{align}\label{domain4}
\| ( T^n- T_0 ^n ) T \Psi \| \le \e \| T_0 ^n T \Psi \| + C_\e \| T \Psi \| , \q \Psi \in D(T_0 ^{n+1}) ,
\end{align}
where $ C_\e >0 $ is a constant depending on $ \e $ and $ n $. In the same manner as in the proof of Lemma \ref{prelimLem1}, one can see that $ T_0 ^n T_1 $ is infinitesimally $ T_0 ^{n+1} $-bounded on $ D $. Combining this with \eqref{domain10}, \eqref{domain2}, \eqref{domain4} and the condition (II), we see that $ T^{n+1} - T_0 ^{n+1} $ is infinitesimally $ T_0 ^{n+1} $-bounded. Hence, it follows from the Kato-Rellich theorem that $ T_0 ^{n+1} + (T^{n+1} - T_0 ^{n+1}) $ is self-adjoint on $ D(T_0 ^{n+1}) $. On the other hand, by the definition of the sum operator, we have the inclusion relation $ T^{n+1} \supset T_0 ^{n+1} + (T^{n+1} - T_0 ^{n+1}) $. Since both sides are self-adjoint, we obtain the operator equality
\begin{align}
T^{n+1} = T_0 ^{n+1} + (T^{n+1} - T_0 ^{n+1}),
\end{align}
which implies \eqref{domain1} for $ n+1 $. Thus the induction step is complete, and the assertion follows. 
\end{proof}

\begin{Lem}\label{prelimLem2} Assume Hypotheses \ref{B-Hypo} and \ref{UV-Hypo}. Then, the following (i)-(iii) hold:
\begin{enumerate}[(i)]
\item $ H_\rm{fr} $ and $ H_\rm{int} $ leave $ \F _\rm{fin} $ invariant.

\item For each $ n \in \N $, $ \F _\rm{fin} $ is a core of $ H_\rm{fr} ^n $.

\item For all $ n \in \N $, $ \rm{ad} _{H_\rm{fr}} ^n (H_\rm{int}) $ is infinitesimally $ H_\rm{fr} $-bounded on $ \F _\rm{fin} $, i.e., for all $ \e >0 $, there exists a constant $ C_\e \ge 0 $ such that for all $ \Psi \in \F _\rm{fin} $,
\begin{align}
\| \rm{ad} _{H_\rm{fr}} ^n (H_\rm{int} ) \Psi \| \le \e \| H_\rm{fr} \Psi \| + C_\e \| \Psi \| . 
\end{align}
\end{enumerate}
\end{Lem}

\begin{proof}
\begin{enumerate}[(i)]
\item It is obvious that $ H_\rm{fr} $ leaves $ \F _\rm{fin} $ invariant from the definition of $ \F _\rm{fin} $ \eqref{Ffin}. The remaining claim follows from Lemma \ref{QEDass1} (i) and (ii).

\item This follows from the general theory of the functional calculus.

\item For each integer $ n \ge 0 $, we define linear operators $ H_\rm{I} ^{(n)} $ and $ H_\rm{II} ^{(n)} $ by
\begin{align}
& D(H_\rm{I} ^{(n)}) := D((I \otimes H_\rm{ph}) ^{1/2}) , \no \\ 
& H_\rm{I} ^{(n)} \Psi := e \int _{\R ^3} d\xx \,  H_\rm{I} ^{(n)} (\xx ) \Psi , \q \Psi \in D(H_\rm{I} ^{(n)}) \label{In} , \\
& D(H_\rm{II} ^{(n)}) := \F _\rm{tot} , \no \\ 
& H_\rm{II} ^{(n)} := e \int _{\R ^3 \times \R ^3} d\xx d\yy  \,  H_\rm{II} ^{(n)} (\xx , \yy ) , \label{IIn}
\end{align}
with
\begin{align}
& H_\rm{I} ^{(n)} (\xx ) := \chisp (\xx ) \sum _{i=1} ^3 \sum _{\substack{ n_1 + n_2 =n , \\ n_1 , n_2 \ge 0 }} \frac{n!}{n_1 ! n_2!} \Normal{ j^i {}^{(n_1 )} (0, \xx ) } \otimes A_i^{(n_3)} (0, \xx ) , \\
& H_\rm{II} ^{(n)} (\xx , \yy ) := \chisp (\xx ) \chisp (\yy ) V_C (\xx - \yy ) \sum _{\substack{ n_1 + n_2 =n , \\ n_1 , n_2 \ge 0 }} \frac{n!}{n_1 ! n_2!}  \Normal{ j^0 {}^{(n_1)} (0, \xx ) j^0 {}^{(n_2)} (0, \yy )} \otimes I , \\
& j^\mu {}^{(n)} (0, \xx ) := \sum _{\substack{ n_1 + n_2 =n , \\ n_1 , n_2 \ge 0 }} \frac{n!}{n_1 ! n_2!} \sum _{l,l' =1} ^4 \psi _l ^{(n_1)} (0, \xx ) ^* \alpha _{ll' } ^\mu \psi _{l'} ^{(n_2)} (0, \xx ) , \\ 
& \psi _l ^{(n)} (0, \xx ) := \sum _{s=\pm 1/2} \Big( b_s \big( (iE_M) ^n \hat{\chiel ^\xx } \, (u_s ^l) ^* \big) + d_s ^* \big( (iE_M) ^n \hat{\chiel ^\xx } \; \widetilde{v} _s ^l \big) \Big) , \\
& A_i ^{(n)} (0, \xx ) := \sum _{r=1,2} \Big( a^{(r)} \Big( \frac{ (i\omega )^n \hat{\chiph ^\xx } e^{(r)}_i }{\sqrt{2\omega }} \Big) + a^{(r)} {}^* \Big( \frac{ (i\omega )^n \hat{\chiph ^\xx } e^{(r)}_i }{\sqrt{2\omega }} \Big) \Big) ,
\end{align}
where the integral in \eqref{In} is taken in the sense of the strong Bochner integral, and the integral in \eqref{IIn} is the Bochner integral with respect to the operator norm. Then, $ H_\rm{I} = H_\rm{I} ^{(0)} , \; H_\rm{II} = H_\rm{II} ^{(0)} $. In the same way as Lemma \ref{s.a.Lem}, one can show that each $ \overline{ H_\rm{I} ^{(n)} } $ is infinitesimally $ H_\rm{fr} $-bounded, and each $ H_\rm{II} ^{(n)} $ is bounded. 

To prove the claim, it is sufficient to show that
\begin{align}\label{adI1}
\rm{ad} _{iH_\rm{fr}} ^n (H_\rm{int}) \Psi = (H_\rm{I} ^{(n)} + H_\rm{II} ^{(n)} ) \Psi , \q \Psi \in \F _\rm{fin} .
\end{align}
The left-hand side can be rewritten as $ \rm{ad} _{i H_\rm{fr} E_{H_\rm{fr}} ([0,E]) } ^n (H_\rm{int}) \Psi $ for sufficiently large $ E \ge 0 $; $ H_\rm{fr} E_{H_\rm{fr}} ([0,E]) $ is bounded. Hence, we have
\begin{align}
& \rm{ad} _{iH_\rm{fr} } ^n (H_\rm{int}) \Psi  \no \\
& = e \int _{\R ^3} d\xx \, \rm{ad} _{iH_\rm{fr} } ^n ( H_\rm{I} ^{(0)} (\xx )) \Psi + e \int _{\R ^3 \times \R ^3} d\xx d\yy  \, \rm{ad} _{iH_\rm{fr} } ^n ( H_\rm{II} ^{(0)} (\xx , \yy ) ) \Psi .
\end{align}
Using Lemmas \ref{boson-comm} and \ref{fermion-comm}, we have
\begin{align}
\rm{ad} _{iH_\rm{fr} } ^n ( H_\rm{I} ^{(0)} (\xx )) \Psi & = H_\rm{I} ^{(n)} (\xx ) \Psi , \\
\rm{ad} _{iH_\rm{fr} } ^n ( H_\rm{II} ^{(0)} (\xx , \yy )) \Psi & = H_\rm{II} ^{(n)} (\xx , \yy ) \Psi .
\end{align}
Therefore, we obtain \eqref{adI1}, and the assertion follows. 
\end{enumerate}
\end{proof}

\begin{Lem}\label{preD1} Under Hypotheses \ref{B-Hypo} and \ref{UV-Hypo}, it follows that
\begin{align}
D(H_\rm{tot} ^n) = D(H_\rm{fr} ^n) .
\end{align}
for all $ n \in \N $.
\end{Lem}

\begin{proof} By Lemma \ref{prelimLem2}, we can apply Lemma \ref{prelimLem1} to the case $ T_0 = H_\rm{fr} $, $ T_1 =H_\rm{int} $ and $ D = \F _\rm{fin} $, and thus the assertion follows.
\end{proof}

\begin{Lem}\label{preD2} Under Hypotheses \ref{B-Hypo}, the following (i) and (ii) hold.
\begin{enumerate}[(i)]
\item For each integer $ n \ge 0 $, $ I \otimes A_j (0,\xx ) $ maps $ D(H_\rm{tot} ^{n+1}) $ into $ D(H_\rm{tot} ^{n}) $.

\item For each integer $ n \ge 0 $, $ \psi _l (0,\xx ) \otimes I $ and $ \psi _l (0,\xx ) ^* \otimes I $ leave $ D(H_\rm{tot} ^{n}) $ invariant.
\end{enumerate}
\end{Lem}

\begin{proof}
\begin{enumerate}[(i)]
\item Applying Lemma \ref{prelimLem1} to the case $ T_0 = H_\rm{fr} $, $ T_1 = I \otimes A_j (0,\xx ) $ and $ D = \F _\rm{fin} $, we see that $ I \otimes A_j (0,\xx ) $ maps $ D(H_\rm{fr} ^{n+1/2}) $ into $ D(H_\rm{fr} ^{n}) $. Combining this with Lemma \ref{preD1}, the assertion follows.

\item Similar to the proof of (i).
\end{enumerate}
\end{proof}

Now we are ready to prove the Gell-Mann -- Low formula. We assume the following:

\begin{Hypo}\label{GS-Hypo} (I) $ H_\rm{tot} $ has a unique ground state $ \Omega $ $ (\| \Omega \| =1 ) $.

(II) $ \Expect{ \Omega _\rm{tot} , \Omega _0 } \neq 0 $, where $ \Omega _0 := \Omega _\ff \otimes \Omega _\bb , \; \Omega _\ff := \{ 1,0,0,... \} \in \F _\ff (\H _\rm{el}) $, and $ \Omega _\bb := \{ 1,0,0,... \} \in \F _\bb (\H _\rm{ph}) $.
\end{Hypo}

For conditions for Hypothesis \ref{GS-Hypo} to hold, see \cite{MR2541206}. Because of some technical problems, the coupling constant $ e $ is currently restricted to a sufficiently small region in order to prove the existence of the ground state. 


Let $ \phi ^{(k)} (0,\xx ) \; (k=1,..., m , \; m \ge 1 , \; \xx \in \R ^3 ) $ denote the point-like field operators, that is, for each $ k $, $ \phi ^{(k)} (0,\xx ) $ denotes $ I \otimes A_j (0, \xx  ) $, $ \psi _l (0,\xx ) \otimes I $, or $ \psi _l (0,\xx ) ^* \otimes I $. For each $ z\in \C $, we set
\begin{align}
\phi _\rm{int} ^{(k)} (z , \xx ) := e^{iz H_\rm{fr}} \phi ^{(k)} (0, \xx ) e^{-iz H_\rm{fr}} . 
\end{align}

\begin{Thm}\label{GML-QED} Assume Hypotheses \ref{B-Hypo}-\ref{GS-Hypo}. Let $ z_1 ,..., z_m \in \C$ with
$\im z_1\le \dots \le \im z_m$, and $ \xx _1 , ... , \xx _m \in \R ^3 $. Choose a simple curve $\Gamma_T^\e$ from $-T(1-i\e)$ to $T(1-i\e)$ ($T,\e>0$) on which
$z_1\succ \dots \succ z_m$. Then, $m$-point Green's function 
\begin{align}
G_m(z_1,\dots,z_m) := e^{i(z_1-z_m)E_0(H_\rm{tot})} \Expect{\Omega, \phi ^{(1)} (0, \xx _1 ) e^{ -i (z_1-z_2) H_\rm{tot} }  \dots \phi ^{(m-1)} (0, \xx _{m-1} ) e^{ -i (z_{m-1}-z_m) H_\rm{tot} } \phi ^{(m)} (0, \xx _m ) \Omega} ,
\end{align}
is well-defined and satisfies the formula
\begin{align}
G_m(z_1,\dots,z_m)
 = \lim _{ T \to \infty } \frac{ \Expect{ \Omega _0 , T\phi ^{(1)} _\rm{int} (z_1 , \xx _1 ) \dots \phi ^{(m)} _\rm{int} (z_m , \xx _m )
\exp \left(-i\int_{\Gamma_T^\e}d\zeta\, H_\rm{int} (\zeta)\right) \Omega _0 } }{\Expect{\Omega _0 , 
T\exp \left(-i\int_{\Gamma_T^\e}d\zeta\, H_\rm{int} (\zeta)\right)\Omega _0}  }. \label{QED-GMLeq}
\end{align}
\end{Thm}


\begin{proof} We have only to see that the conditions of Theorem \ref{GML} hold when $ H_0 = H_\rm{fr} $, $ H_1 = H_\rm{int} $, $ H = H_\rm{tot} $ and $ A_k = \phi ^{(k)} (0, \xx _k ) $. 

As is well known, $ H_\rm{fr} $ has a unique ground state $ \Omega _0 $, and the corresponding eigenvalue is zero. Thus, Assumption \ref{Ass3} (I) holds. Assumption \ref{Ass3} (II) and (III) follow from Hypothesis \ref{GS-Hypo}.

From Lemma \ref{field-spec} (iii) and (iv), each $ \phi ^{(k)} (0, \xx _k ) $ is in $ \mathcal{C}_0 $-class. The remaining assumptions follow from Lemma \ref{preD2}. Therefore, the desired result follows.
\end{proof}

\begin{Rem}\normalfont The above formula \eqref{QED-GMLeq} is more general than the Gell-Mann -- Low formula 
discussed in physics literatures. To obtain the original Gell-Mann -- Low formula, we regard the arguments $ z_k \in \C \; (k=1,...,m) $ as the time parameters which are usually real numbers, $ z_k \in \R $. Then, these are 
naturally time-ordered in $ \R $ whenever these are different from each other. Therefore, to derive the original formula,
choose a simple curve $ \Ga _T ^\e $ from $ -T (1-i\e ) $ to $ T (1-i\e ) $
in such a way that this natural time-ordering coincides our time-ordering defined above. For instance, take a polyline that
passes $ -T (1-i\e ) $, $t_{\rm{min}}$, $t_{\rm{max}}$, and $ T (1-i\e ) $ in this order, where 
\[ t_{\rm{min}}=\min\{t_1,\dots,t_m\},\q t_{\rm{max}}=\max\{t_1,\dots,t_m\}. \]
\end{Rem}

\section*{Acknowledgements}
The authors are grateful to Professor Asao Arai for continuous encouragements, comments,
discussions, and critical reading of the manusprict. 
They also thank Dr. Kazuyuki Wada and Dr. Daiju Funakawa for discussions and comments.

\bibliographystyle{habbrv}

\bibliography{ref}

\end{document}